\documentclass[preprint]{imsart}

\RequirePackage{amsthm,amsmath,amsfonts,amssymb}
\RequirePackage[authoryear]{natbib}%% uncomment this for author-year citations
\RequirePackage[colorlinks,citecolor=blue,urlcolor=blue]{hyperref}%% uncomment this for coloring bibliography citations and linked URLs

\startlocaldefs

\newtheorem{theo}{Theorem}[section]
\newtheorem{lem}[theo]{Lemma}
\newtheorem{rem}[theo]{Remark}
\newtheorem{cor}[theo]{Corollary}
\newtheorem{proposition}[theo]{Proposition}

\newcommand{\LA}{{$\mathbf{{\cal A}_n}$}}
\newcommand{\LAI}{{$\mathbf{{\cal A}_{n,1}}$}}
\newcommand{\LAII}{{$\mathbf{{\cal A}_{n,2}}$}}
\newcommand{\LB}{{$\mathbf{{\cal B}_n}$}}
\newcommand{\LC}{{$\mathbf{{\cal C}_n}$}}
\newcommand{\LD}{{$\mathbf{{\cal D}_n}$}}
\newcommand{\LDS}{{$\mathbf{{\cal D}_n^{*}}$}}
\newcommand{\LDI}{{$\mathbf{{\cal D}_{n,1}}$}}
\newcommand{\LDII}{{$\mathbf{{\cal D}_{n,2}}$}}
\newcommand{\LE}{{$\mathbf{{\cal E}_n}$}}
\newcommand{\LF}{{$\mathbf{{\cal F}_n}$}}
\newcommand{\LG}{{$\mathbf{{\cal G}_n}$}}
\newcommand{\LH}{{$\mathbf{{\cal H}_n}$}}
\newcommand{\LI}{{$\mathbf{{\cal I}_n}$}}
\newcommand{\LJ}{{$\mathbf{{\cal J}_n}$}}
\newcommand{\LK}{{$\mathbf{{\cal K}_n}$}}
\newcommand{\LL}{{$\mathbf{{\cal L}_n}$}}
\newcommand{\LM}{{$\mathbf{{\cal M}_n}$}}
\newcommand{\LN}{{$\mathbf{{\cal N}_n}$}}
\newcommand{\LO}{{$\mathbf{{\cal O}_n}$}}
\newcommand{\LP}{{$\mathbf{{\cal P}_n}$}}
\newcommand{\LQ}{{$\mathbf{{\cal Q}_n}$}}
\newcommand{\LR}{{$\mathbf{{\cal R}_n}$}}
\newcommand{\LS}{{$\mathbf{{\cal S}_n}$}}

\numberwithin{equation}{section}
\endlocaldefs

\begin{document}

\begin{frontmatter}
\title{Asymptotic equivalence for nonparametric additive regression}
%\title{A sample article title with some additional note\thanksref{T1}}
\runtitle{Asymptotic equivalence for nonparametric additive regression}
%\thankstext{T1}{A sample of additional note to the title.}

\begin{aug}
%%%%%%%%%%%%%%%%%%%%%%%%%%%%%%%%%%%%%%%%%%%%%%%
%% Only one address is permitted per author. %%
%% Only division, organization and e-mail is %%
%% included in the address.                  %%
%% Additional information such as            %%
%% identifying the corresponding author must %%
%% be included in in the Acknowledgments     %%
%% section if necessary.                     %%
\author[A]{\fnms{Moritz}~\snm{Jirak}\ead[label=e1]{moritz.jirak@univie.ac.at}},
\author[B]{\fnms{Alexander}~\snm{Meister}\ead[label=e2]{alexander.meister@uni-rostock.de}}
\and
\author[C]{\fnms{Angelika}~\snm{Rohde}\ead[label=e3]{angelika.rohde@stochastik.uni-freiburg.de}}
%%%%%%%%%%%%%%%%%%%%%%%%%%%%%%%%%%%%%%%%%%%%%%
%% Addresses                                %%
%%%%%%%%%%%%%%%%%%%%%%%%%%%%%%%%%%%%%%%%%%%%%%
\address[A]{Institut f\"ur Statistik und Operations Research, Universit\"at Wien, Oskar-Morgenstern-Platz 1, 1090 Wien, Austria\printead[presep={,\ }]{e1}}

\address[B]{Institut f\"ur Mathematik, Universit\"at Rostock, Ulmenstra{\ss}e 69, 18057 Rostock, Germany\printead[presep={,\ }]{e2}}

\address[C]{Mathematisches Institut, Albert-Ludwigs-Universit\"at Freiburg, Ernst-Zermelo-Stra{\ss}e 1, 79104 Freiburg, Germany\printead[presep={,\ }]{e3}}
\end{aug}

\begin{abstract}
We prove asymptotic equivalence of nonparametric additive regression and an appropriate Gaussian white noise experiment in which a multidimensional shifted Wiener process is observed, whose dimension equals the number of additive components. The shift depends on the additive components of the regression function and  solely the one- and two-dimensional marginal distributions of the covariates via an explicitly specified bounded but non-compact linear operator~$\Gamma$. The number of additive components $d$ is allowed to increase moderately with respect to the sample size. In the special case of pairwise independent components of the covariates, the white noise model decomposes into $d$ independent univariate processes. Moreover, we study approximation in some semiparametric setting where $\Gamma$ splits into a multiplication operator and an asymptotically negligible Hilbert-Schmidt operator.
\end{abstract}

\begin{keyword}[class=MSC2020]
\kwd[Primary ]{62G08}
\kwd[; secondary ]{62B15}
\end{keyword}

\begin{keyword}
\kwd{Additive model}
\kwd{asymptotic sufficiency}
\kwd{high-dimensional regression}
\kwd{Le Cam distance}
\kwd{white noise models}
\end{keyword}

\end{frontmatter}
%%%%%%%%%%%%%%%%%%%%%%%%%%%%%%%%%%%%%%%%%%%%%%
%% Please use \tableofcontents for articles %%
%% with 50 pages and more                   %%
%%%%%%%%%%%%%%%%%%%%%%%%%%%%%%%%%%%%%%%%%%%%%%
%\tableofcontents

%%%%%%%%%%%%%%%%%%%%%%%%%%%%%%%%%%%%%%%%%%%%%%
%%%% Main text entry area:

\section{Introduction} \label{1}
%%%%%%%%%%%%%%%%%%%%%%%%%%%%%%%%

The classical  nonparametric regression model is described by the statistical experiment with i.i.d.~observations $(X_j,Y_j)$, $j=1,\ldots,n$, that satisfy \begin{equation} \label{eq:1.1}
Y_j \, = \, g(X_j) \, + \, \varepsilon_j\,, \qquad j=1,\ldots,n\,
\end{equation}
for some unknown regression function $g : \mathbb{R}^d \to \mathbb{R}$. Here, the covariates $X_j$ take their values in $\mathbb{R}^d$, $d\geq 1$, and the random regression errors $\varepsilon_j$ fulfill $\mathbb{E}(\varepsilon_j\mid X_j)=0$ and $\mathbb{E}(\varepsilon_j^2\mid X_j)\leq \sigma^2$ almost surely. Throughout this work, we restrict to the standard model in which the $\varepsilon_j$ are assumed to be independent of the $X_1,\ldots,X_n$ and normally distributed with expectation $0$ and fixed variance $\sigma^2$. The regression function $g$ represents the parameter of interest while the $X_j$ possess  a known $d$-variate bounded Lebesgue density $p_X$ which is bounded away from zero on its support $[0,1]^d$.

It is well-known that  under smoothness constraints on $g$, nonparametric estimators  suffer from a so-called curse of dimensionality; precisely their convergence rate equals $n^{-2\beta/(2\beta+d)}$ where $\beta$ denotes the smoothness level of $g$. These rates become slower when the dimension $d$ gets larger. As they are known to be minimax optimal, the only way out of this curse is to impose stronger conditions on $g$. A common way to overcome the slow rates is to restrict to additive models, i.e. the regression function $g$ is assumed to have the structure
\begin{equation} \label{eq:1.2}
g(x) \, = \, \sum_{\ell=1}^d g_\ell(x_\ell)\,, \qquad x=(x_1,\ldots,x_d)\,,
\end{equation}
where the $g_\ell : \mathbb{R}\to\mathbb{R}$ are unknown smooth functions. Indeed, if $g$ is of the form \eqref{eq:1.2}, then all $g_{\ell}$ (and $g$ in particular) can be estimated with their univariate minimax rate $n^{-2\beta/(2\beta+1)}$, provided some mild identifiability requirement. Early interest goes back to the papers of \cite{S1985} and \cite{S1986}, in which these convergence rates were established. Nonparametric estimation methods involving backfitting  were proposed and studied in \cite{L1997}, \cite{HM2004}, and \cite{MP2006}, and for generalized additive models in \cite{YPM2008}. Issues of smoothing parameter selection were considered in e.g. \cite{MP2005}. Optimality results for estimators can be found in \cite{G92}, \cite{HKM2006} and \cite{HM2011}. Moreover, extensions of the standard problem have been addressed, e.g.~to quantile regression in \cite{LMP2010}, recently to non-Euclidean data which take their values in general Hilbert spaces in \cite{JP20}, \cite{JPK21A}, \cite{JPK21B} and \cite{JLMP2022}, to matrix-valued data and Lie groups in \cite{LMP23}, to high-dimensional but sparse models in \cite{GMW2021} and to stochastic optimization in \cite{AT2025}.

To the best of our knowledge, an asymptotically equivalent white noise experiment for nonparametric additive regression models in Le Cam's sense has been missing so far. In nonparametric Le Cam theory, a major goal is to introduce a specific Wiener process, in which the target function occurs in the drift, such that observing this process is asymptotically equivalent to the original nonparametric experiment. Asymptotic equivalence means that the Le Cam distance between the experiments converges to zero as the sample size $n$ tends to infinity. Once this has been shown, basically all asymptotic properties which have been established in one experiment can be taken over to the other experiment. For the basic concept we refer to the books of \cite{LC1986} and \cite{LCY2000}. To explain it briefly: let $P_{X,\theta}$ and $P_{Y,\theta}$, $\theta \in \Theta$, be the competing probability measures of two statistical experiments with a joint parameter space $\Theta$. If two Markov kernels $K_{X,Y}$ and $K_{Y,X}$ exist such that
\begin{align*}
P_{Y,\theta} & \, = \, \int K_{X,Y}(x,\bullet) \, dP_{X,\theta}(x)\,, \\
P_{X,\theta} & \, = \, \int K_{Y,X}(y,\bullet) \, dP_{Y,\theta}(y)\,,
\end{align*}
for all $\theta\in\Theta$, then the experiments are called equivalent. Therein, the Markov kernels must not depend on $\theta$. Now assume that there are two sequences of such experiments indexed by $n$, which can usually be interpreted as the sample size in at least one of the experiments. Then these sequences are called asymptotically equivalent if the total variation distance $\mbox{TV}(\cdot,\cdot)$ satisfies
\begin{align*}
\lim_{n\to\infty} \, \sup_{\theta\in\Theta} &\, \mbox{TV}\Big(P_{Y,\theta,n} \, , \, \int K_{X,Y,n}(x,\bullet) \, dP_{X,\theta,n}(x)\Big) \, = \, 0\,, \\
\lim_{n\to\infty} \, \sup_{\theta\in\Theta} &\, \mbox{TV}\Big(P_{X,\theta,n} \, , \, \int K_{Y,X,n}(y,\bullet) \, dP_{Y,\theta,n}(y)\Big) \, = \, 0\,.
\end{align*}
By a slight abuse of language, we call the experiments asymptotically equivalent rather than the underlying sequences.

For standard nonparametric regression with $d=1$, the first asymptotic equivalence proof has been given in \cite{BL1996}. Other papers have considered the impact on convergence rates (\cite{R2004}), robust nonparametric regression (\cite{CZ2009}) and extensions to non-Gaussian regression errors (\cite{GN1998}, \cite{GN2002}, \cite{CZ2010}), to functional linear regression (\cite{M2011}), to non-regular regression errors with a Poisson point process as the counterpart experiment (\cite{MR2013}), to unknown error variance (\cite{C2007}), to fractional noise (\cite{SH2014}) and to spherical covariates under specific location of these covariates in the fixed design setting (\cite{K2025}). We also mention the paper on the information contained in additional observations (\cite{M1986}) and the fundamental paper for asymptotic equivalence for density estimation (\cite{N1996}). In \cite{C2006} the setting of bivariate covariates ($d=2$) is studied; while \cite{R2008} considers general fixed dimension $d$ under the uniform design density and smoothness classes of periodic functions. Thus, this latter paper also includes additive regression models; however, the smoothness level is assumed to be large with respect to the dimension. Concretely, $\beta>d/2$ is imposed -- transferred to our notation. In the current paper our smoothness conditions on each additive component $g_\ell$ are much less restrictive, and $d=d_n$ is allowed to increase moderately with respect to $n$. In particular, the required smoothness level is independent of the dimension. Furthermore, the white noise model in \cite{R2008}, which involves a $d$-dimensional Wiener process and the multivariate regression function $g$ as the shift, seems less appropriate to exploit the specific additive structure of $g$ which we impose in (\ref{eq:1.2}). Therefore, our target model describes the observation of an $\mathbb{R}^d$-valued stochastic process on the one-dimensional domain $[0,1]$.

In the current work, we provide a rigorous asymptotic equivalence result for nonparametric additive regression: we prove that nonparametric additive regression is asymptotically equivalent in Le Cam's sense to a specific white noise experiment in which one observes a Gaussian process
$$ R_n(t) \, = \, \int_0^t \big[\Gamma^{1/2} \tilde{g}\big](s)\, ds \, + \, \frac{\sigma}{\sqrt{n}} \, \tilde{W}(t)\,, \quad t\in [0,1]\,, $$
with a $d$-dimensional Wiener process $\tilde{W}$, the target function $\tilde{g}(t) := (g_1(t),\ldots,g_d(t))^\top$ and an explicitly specified bounded but non-compact linear operator $\Gamma$ which depends on the two-dimensional marginals of the design density $p_X$ (cf.~Theorem \ref{T:1}). The proof extends over Sections \ref{2}--\ref{FT}. Section~\ref{SCE} contains special cases and extensions. In particular, we prove in Subsection~\ref{SCE.1} that for pairwise independent components of the $d$-dimensional covariates $X_j$, the operator $\Gamma$ degenerates such that the observed process can be decomposed into $d$ independent one-dimensional Gaussian processes for each of the centered components of $\tilde{g}$ and a real-valued Gaussian random variable for the constant shift term. In this special case we can also show asymptotic equivalence to the experiment of $d$ independent univariate regression models for each of the centered component functions and the real-valued Gaussian shift variable (cf.~Theorem \ref{T:2}). Moreover, in some semiparametric framework in Subsection~\ref{SCE.2}, we show that the operator $\Gamma$ can asymptotically be approximated by the componentwise multiplication operator without independence assumptions on the components of the covariates (cf.~Theorem \ref{T:SCE.1}). This is useful to deduce asymptotic minimax lower bounds in plenty of settings
of nonparametric curve estimation. Finally, we address the question whether the underlying statistic can also be approximated under unknown design density (cf.~Remark~\ref{rem: 6.4}).

\section{Regression model and pilot approximation} \label{2}
%%%%%%%%%%%%%%%%%%%%%%%%%%%%%%%%%%%%%%%%%%%%%%%%%%%%%%%%%

%\hyperref[A]{\LA}\label{A}
%\LA \label{A}: The original regression experiment \hyperref[A]{\LA
Our starting point is the following nonparametric additive regression experiment:
\hyperref[A]{\LA}\label{A} describes the observation of $(X_j,Y_j)$, $j=1,\ldots,n$, where (\ref{eq:1.1}) holds true and the regression function $g$ has the shape (\ref{eq:1.2}). All additive components $g_\ell$, $\ell=1,\ldots,d$, are assumed to satisfy the H\"older condition
\begin{align} \nonumber  |g_\ell(x) - g_\ell(y)| & \, \leq \, C \cdot |x-y|^\beta\,, \qquad \forall x,y\in [0,1]\,, \\
\label{eq:Hoelder} |g_\ell(x)| & \, \leq \, C \,, \qquad \forall x\in [0,1]\,, \end{align}
for some global constants $C>0$ and $\beta\in (0,1]$. Thus, the parameter space $\Theta$ contains all functions $\tilde{g}$ with $\tilde{g}(t) := \big(g_1(t),\ldots,g_d(t)\big)^\top$, $t\in [0,1]$, where all $g_\ell$, $\ell=1,\ldots,d$, fulfill (\ref{eq:Hoelder}). The design density $p_X$ of $X_j$ is assumed to be known and therefore excluded from the parameter space. We assume $p_X(x)=0$ for all $x\not\in [0,1]^d$ and $p_X(x) \in [\rho,1/\rho]$ for all $x\in [0,1]^d$ for some fixed global constant $\rho>0$ but we do not impose any kind of smoothness condition on $p_X$.\pagebreak

Our first goal is to approximate \hyperref[A]{\LA} by an experiment with finite-dimensional parameter space. To this aim, we introduce the basis functions
$$ \varphi_{n,k,k'}(x) \, := \, {\bf 1}_{[k/K_n,(k+1)/K_n)}(x_{k'})\,, \qquad x=(x_1,\ldots,x_d)\,, $$
for $k=0,\ldots,K_n-1$, $k'=1,\ldots,d$, with some integer smoothing parameter $K_n$ still to be chosen. Note that $\varphi_{n,k,k'}(x)$ only affects the $k'$-th component, while the others remain untouched. Define the linear hull
$$ \Phi_n \, := \mbox{span}\big(\varphi_{n,k,k'} \, : \, k=1,\ldots,K_n, \, k'=1,\ldots,d\big)\,, $$
as a subspace of $L_2(p_X)$, that is, the Hilbert space which contains all measurable functions $f$ on the domain $[0,1]^d$ with the norm
$$ \|f\|_{p_X} \, := \, \Big(\int |f(x)|^2 \, p_X(x) \, dx\Big)^{1/2}\, < \, \infty\,. $$
We remark that the linear hull of the functions $\varphi_{n,0} :\equiv 1$ and $\varphi_{n,k,k'}$, $k=1,\ldots,K_n-1$, $k'=1,\ldots,d$, also coincides with $\Phi_n$ since $\varphi_{n,0} = \sum_{k=0}^{K_n-1} \varphi_{n,k,k'}$ for all $k'=1,\ldots,d$. On the other hand, these functions are linearly independent (see the proof Lemma \ref{L:ON}), and hence $\Phi_n$ has dimension $$K_n^* :=  1+d\cdot(K_n-1).$$ Let $\psi_{n,1},\ldots,\psi_{n,K_n^*}$ denote a specific orthonormal basis of $\Phi_n$ with respect to the inner product $\langle\cdot,\cdot\rangle_{p_X}$ of $L_2(p_X)$, satisfying the inequality in the following Lemma \ref{L:ON}. Its proof is deferred to Section \ref{sec: 8}.

\begin{lem} \label{L:ON}
There exists an orthonormal basis $\psi_{n,j}$, $j=1,\ldots,K_n^*$, of $\Phi_n$ with respect to the inner product of $L_2(p_X)$ such that
$$ \sum_{j=1}^{K_n^*} \psi_{n,j}^2(x) \, \leq \, \rho^{-1} \cdot \big\{1 + K_n\cdot d \cdot (1 + \pi^2/6)\big\}\,, \qquad \forall x\in [0,1]^d\,.$$
\end{lem}

Note that, for $K_n\geq 2$ and fixed dimension $d$, the upper bound in Lemma \ref{L:ON} is of magnitude ${\cal O}(K_n^*)$, where the constant factor depends only on $\rho$. By $g^{[K_n^*]}$, we denote the orthogonal projection of $g$ onto $\Phi_n$ in the Hilbert space $L_2(p_X)$; thus
$$ g^{[K_n^*]} \, = \, \sum_{k=1}^{K_n^*} \langle g , \psi_{n,k}\rangle_{p_X} \cdot \psi_{n,k}\,. $$

Now we introduce the experiment \hyperref[B]{\LB}\label{B}, which coincides with \hyperref[A]{\LA} except that the regression function $g$ is replaced by $g^{[K_n^*]}$.

\smallskip
In order to keep the conditions as general as possible, the dimension $d=d_n$ of the covariates is allowed to potentially increase in $n$.
\begin{proposition}
    The experiments  \hyperref[A]{\LA} and \hyperref[B]{\LB} are asymptotically equivalent if\begin{equation} \label{eq:AB}
\lim_{n\to\infty}\, d_n\cdot n\cdot K_n^{-2\beta} \, = \, 0\,.
\end{equation}
\end{proposition}

\begin{proof}
The Le Cam distance between \hyperref[A]{\LA} and \hyperref[B]{\LB} converges to zero whenever the expected squared Hellinger distance between the $n$-dimensional Gaussian distributions (conditional on the design) $N\big(\{g(X_j)\}_{j} , \sigma^2 I_n\big)$ and $N\big(\{g^{[K_n^*]}(X_j)\}_{j}, \sigma^2 I_n\big)$ does so uniformly with respect to the target parameter $g$. Here, $I_n$ denotes the $n$-dimensional identity matrix, and we write $\mathbb{H}(\cdot,\cdot)$ for the Hellinger distance. This can be seen as follows: the total variation distance $\mbox{TV}(\mathbb{P}_{X,Y},\mathbb{P}_{X,Z})$ between the probability measures $\mathbb{P}_{X,Y}$ and $\mathbb{P}_{X,Z}$ of some random variables $(X,Y)$ and $(X,Z)$, respectively, can be written as
\begin{align*} \frac12\sup_{\|f\|_\infty=1} \big|\mathbb{E} f(X,Y) - \mathbb{E}f(X,Z)\big| & \, \leq \, \frac12 \mathbb{E} \sup_{\|f\|_\infty=1} \big|\mathbb{E}[f(X,Y)|X] - \mathbb{E}[f(X,Z)|X]\big| \\ & \, = \, \frac12 \mathbb{E}\, \mbox{TV}\big(\mathbb{P}_{Y|X},\mathbb{P}_{Z|X}\big) \\ & \, \leq \, \big\{\mathbb{E}\, \mathbb{H}^2\big(\mathbb{P}_{Y|X},\mathbb{P}_{Z|X}\big)\big\}^{1/2}\,.
\end{align*}
Specifically,
\begin{align} \nonumber
\sup_g & \, \mathbb{E}\, \mathbb{H}^2\big\{N\big(\{g(X_j)\}_{j=1,\ldots,n} , \sigma^2 I_n\big)\, , \, N\big(\{g^{[K_n^*]}(X_j)\}_{j=1,\ldots,n} , \sigma^2 I_n\big)\} \\ \nonumber
& \, \leq \, 2\, \sup_g \, \Big\{ 1 - \exp\Big(- \frac1{8\sigma^2} \sum_{j=1}^n \mathbb{E} \big|g(X_j) - g^{[K_n^*]}(X_j)\big|^2\Big)\Big\} \\
& \label{eq:Hell}\, = \, 2\, \sup_g \, \Big\{ 1 - \exp\Big(- \frac{n}{8\sigma^2} \big\|g - g^{[K_n^*]}\big\|_{p_X}^2\Big)\Big\}
%\label{eq:Hell}
%& \, \leq \, 2\cdot \Big\{1  -  \exp\Big( - \frac1{8 \sigma^2\rho} \cdot C^2\cdot d \cdot n \cdot K_n^{-2\beta}\Big)\Big\}\,,
\end{align}
by  Jensen's inequality.
As $g^{[K_n^*]}$ is the best approximation of $g$ within $\Phi_n$, we obtain together with \eqref{eq:1.2} and \eqref{eq:Hoelder},
\begin{align} \nonumber
\big\| g& - g^{[K_n^*]}\big\|_{p_X}^2 \\ \nonumber & \, \leq \, \idotsint \Big|\sum_{j=1}^d g_j(x_j) \, - \, \sum_{j=1}^{d} \sum_{k=0}^{K_n-1} K_n \cdot\int_{k/K_n}^{(k+1)/K_n} g_j(y_j) dy_j \cdot  \varphi_{n,k,j}(x)\Big|^2 p_X(x) \, dx \\ \nonumber
& \, \leq \, \frac1\rho \, \sum_{j=1}^d \sum_{k=0}^{K_n-1} \int_{k/K_n}^{(k+1)/K_n} \Big|g_j(x_j) - K_n\cdot \int_{k/K_n}^{(k+1)/K_n} g_j(y_j) dy_j\Big|^2 dx_j \\ \label{eq:2.2}
& \, \leq \, \frac1\rho\cdot d\cdot C^2 \cdot K_n^{-2\beta}\, .		
\end{align}
Plugging this bound into \eqref{eq:Hell} reveals the claim.
\end{proof}

\section{Dimension reduction} \label{3}
%%%%%%%%%%%%%%%%%%%%%%%%%%%%%%%%

\paragraph*{Sufficiency in \hyperref[B]{\LB}} The conditional likelihood function of the observations $Y_1,\ldots,Y_n$ in the experiment \hyperref[B]{\LB}, given the covariates $X_1,\ldots,X_n$, is equal to
\begin{align*}  f_{B,n}(\tilde{g};y_1,\ldots,y_n)  \, &= \, (2\pi\sigma^2)^{-n/2}\, \prod_{j=1}^n \exp\big\{-\big(y_j - g^{[K_n^*]}(X_j)\big)^2/(2\sigma^2)\big\} \\
 \, &= \, (2\pi\sigma^2)^{-n/2}\, \exp\Big(-\frac1{2\sigma^2}\sum_{j=1}^n y_j^2\Big) \cdot \exp\Big(- \frac1{2\sigma^2} \sum_{j=1}^n \big|g^{[K_n^*]}(X_j)\big|^2\Big) \\ &  \hspace{4.2cm}\cdot \exp\Big(\frac1{\sigma^2}  \sum_{k=1}^{K_n^*} \langle g,\psi_{n,k}\rangle_{p_X}  \sum_{j=1}^n y_j \cdot \psi_{n,k}(X_j)\Big)\,.
\end{align*}
Then the Neyman-Fisher lemma yields that the covariates $X_1,\ldots,X_n$ along with
$$ Z_n \, := \, \Big(\frac1n \sum_{j=1}^n Y_j \cdot \psi_{n,k}(X_j)\Big)_{k=1,\ldots,K_n^*}\, $$
form a sufficient statistic in the experiment \hyperref[B]{\LB}.\\ % \pagebreak

Therefore, \hyperref[B]{\LB} is equivalent to the experiment \hyperref[C]{\LC}\label{C}, in which only this statistic is observed.

\paragraph*{Standardization of \hyperref[C]{\LC}} The conditional distribution of $Z_n$ given $X_1,\ldots,X_n$ is Gaussian with expectation vector $\widehat{M}_n G_n$, and covariance matrix $\sigma^2 \cdot \widehat{M}_n / n$, where
$$ \widehat{M}_n \, := \, \big\{ \langle \psi_{n,k},\psi_{n,k'}\rangle_{X,n} \big\}_{k,k'=1,\ldots,K_n^*}\quad \text{ and }\quad
 G_n \, := \, \{\langle g,\psi_{n,k}\rangle_{p_X}\big\}^\top_{k=1,\ldots,K_n^*}\,, $$
with $\langle f,\tilde{f} \rangle_{X,n} := \frac1n \sum_{j=1}^n f(X_j) \tilde{f}(X_j)$. The fully observed random matrix $\widehat{M}_n$ is symmetric and positive semi-definite, and hence admits the decomposition $\widehat{M}_n = \widehat{U}_n^\top \widehat{\Lambda}_n \widehat{U}_n$, with orthogonal and real-valued matrix $\widehat{U}_n$, and diagonal matrix $\widehat{\Lambda}_n$ with nonnegative entries. Therefore, observing $X_1,\ldots,X_n$ and $Z_n$ is equivalent to observing $X_1,\ldots,X_n$ and
$$ Z_n' \, := \, \widehat{\Lambda}_n \widehat{U}_n G_n \, + \, \frac{\sigma}{\sqrt{n}} \, \widehat{\Lambda}_n^{1/2} \eta_n\,, $$
where $\eta_n$ denotes a $K_n^*$-dimensional random vector with i.i.d.~standard Gaussian components, which is independent of the covariates $X_1,\ldots,X_n$. Those components of $Z_n'$, for which the corresponding diagonal entries of $\widehat{\Lambda}_n$ vanish, are just zero. Instead,  one may generate independent centered Gaussian random variables (ancillary random variables) with variance $\sigma^2/n$ (independently of the $X_1,\ldots,X_n$ and the other components of $Z_n'$). All other components of $Z_n'$ may be divided by the corresponding diagonal entries of $\widehat{\Lambda}_n^{1/2}$ so that, without losing any information on $\tilde{g}$, we arrive at the observation of $X_1,\ldots,X_n$ and
$$ Z_n'' \, = \,  \widehat{\Lambda}_n^{1/2} \widehat{U}_n G_n \, + \, \frac{\sigma}{\sqrt{n}} \, \eta'_n\,, $$
where $\eta'_n$ has the same distribution as $\eta_n$ and is also independent of $X_1,\ldots,X_n$. Note that $\widehat{U}_n^\top Z_n'' = \widehat{M}_n^{1/2} G_n \, + \, \sigma \eta''_n / \sqrt{n}$ with another independent copy $\eta''_n$ of $\eta'_n$. \\

Hence, the experiment \hyperref[C]{\LC} is equivalent to the experiment \hyperref[D]{\LD}\label{D}, in which one observes $X_1,\ldots,X_n$ and -- conditionally on these covariates -- a $K_n^*$-dimensional Gaussian random vector with  mean vector $$\widehat{M}_n^{1/2} G_n$$ and  covariance matrix $\sigma^2 I_{K_n^*} / n$ where we write $I_{K_n^*}$ for the $K_n^*\times K_n^*$-identity matrix.

\section{Localization and removal of the covariates} \label{Loc}
%%%%%%%%%%%%%%%%%%%%%%%%%%%%%%%%%%%%%%%%%%%%%%%%%%%%%%

The next goal is to replace the random matrix $\widehat{M}_n^{1/2}$ by the identity $I_{K_n^*}$ in experiment \hyperref[D]{\LD} and to remove the covariates $X_1,\ldots,X_n$.

\subsection{Localization procedure}
It follows from the definition of $\widehat{M}_n$ that $\mathbb{E} \widehat{M}_n = I_{K_n^*}$, since the $\psi_{n,k}$, $k=1,\ldots,K_n^*$, form an orthonormal system with respect to $\langle\cdot,\cdot\rangle_{p_X}$. To proceed, a localization step is required. Concretely, we split the original experiment \hyperref[A]{\LA} into two independent experiments \hyperref[A1]{\LAI}\label{A1} and \hyperref[A2]{\LAII}\label{A2}, where \hyperref[A1]{\LAI} describes the observation of the regression data $(X_j,Y_j)$, $j=1,\ldots,m := \lfloor n/2\rfloor$, while \hyperref[A2]{\LAII} contains all remaining data $(X_j,Y_j)$, $j=m+1,\ldots,n$.

\smallskip
The transformation and approximation steps of the previous sections may be applied to the data from \hyperref[A1]{\LAI} and \hyperref[A2]{\LAII} separately so that we arrive at the experiment \hyperref[DS]{\LDS}\label{DS}, which consists of two independent parts \hyperref[DI]{\LDI}\label{DI} and \hyperref[DII]{\LDII}\label{DII}:

\smallskip
Under \hyperref[DI]{\LDI}, we observe $X_1,\ldots,X_m$ and the random vector $Z_{n,1}$ which is conditionally Gaussian with the expectation vector $$\widehat{M}_{n,1}^{1/2} G_n$$ and the covariance matrix $\sigma^2 I_{K_n^*} / m$ given $X_1,\ldots,X_m$.
    Accordingly, under \hyperref[DII]{\LDII}, the observations are $X_{m+1},\ldots,X_n$ and $Z_{n,2}$ having the conditional Gaussian distribution with the expectation vector $$\widehat{M}_{n,2}^{1/2} G_n$$ and the covariance matrix $\sigma^2 I_{K_n^*} / (n-m)$ given $X_{m+1},\ldots,X_n$. Therein,
\begin{align*}
\widehat{M}_{n,1} & \, := \, \Big\{\frac1m \sum_{j=1}^m \psi_{n,k}(X_j) \psi_{n,k'}(X_j)\Big\}_{k,k'=1,\ldots,K_n^*}\,, \\
\widehat{M}_{n,2} & \, := \, \Big\{\frac1{n-m} \sum_{j=m+1}^n \psi_{n,k}(X_j) \psi_{n,k'}(X_j)\Big\}_{k,k'=1,\ldots,K_n^*}\, .
\end{align*}
Note that, under the condition (\ref{eq:AB}), the experiments \hyperref[A]{\LA} and \hyperref[DS]{\LDS} are asymptotically equivalent.

\smallskip
Now we introduce a pilot estimator $\widehat{G}_{n,1}$ of $G_n$ in the experiment \hyperref[DI]{\LDI}, whose specification is deferred to Subsection \ref{SS:pilot}. For notational convenience, fix $$\widehat{g}_{n,1} := \sum_{k=1}^{K_n^*} (\widehat{G}_{n,1})_k \cdot \psi_{n,k},
$$
with $(\widehat{G}_{n,1})_k$ the $k$th component of $\widehat{G}_{n,1}$ so that $ G_n - \widehat{G}_{n,1} = \big\{\langle g - \widehat{g}_{n,1} , \psi_{n,k}\rangle_{p_X}\big\}_{k=1,\ldots,K_n^*}\, . $\\

Easily, we find that experiment \hyperref[DS]{\LDS} is equivalent to experiment \hyperref[E]{\LE}\label{E} in which one observes $X_1,\ldots,X_n$, $Z_{n,1}$ and $$Z_{n,2}^* := Z_{n,2} - \widehat{M}_{n,2}^{1/2} \widehat{G}_{n,1} + \widehat{G}_{n,1}$$ by just adding the observed vector $\widehat{G}_{n,1} - \widehat{M}_{n,2}^{1/2} \widehat{G}_{n,1}$ to $Z_{n,2}$ in order to transform from \hyperref[DS]{\LDS} to \hyperref[E]{\LE}; and subtracting it for the reverse transform. \\

Therein, it is essential that $\widehat{G}_{n,1} - \widehat{M}_{n,2}^{1/2} \widehat{G}_{n,1}$ is measurable with respect to the $\sigma$-field generated by $X_1,\ldots,X_n$ and $Z_{n,1}$. \\

Moreover, we define the experiment \hyperref[F]{\LF}\label{F} in which we observe $X_1,\ldots,X_n$, $Z_{n,1}$ and $\zeta_{n,2}$, which obeys the Gaussian distribution with the expectation vector $G_n$ and the covariance matrix $\sigma^2 I_{K_n^*} / (n-m)$ and which is independent of $X_1,\ldots,X_n$ and $Z_{n,1}$.

\begin{proposition}\label{prop: 2}
    The experiments \hyperref[E]{\LE} and \hyperref[F]{\LF} are asymptotically equivalent if
    \begin{equation} \label{eq:EF}
\lim_{n\to\infty} \, K_n^* \cdot \sup_g \, \mathbb{E} \|g - \widehat{g}_{n,1}\|_{p_X}^2 \, = \, 0\, .
\end{equation}
\end{proposition}

\begin{proof}
Using the techniques from (\ref{eq:Hell}), the Le Cam distance between the experiments \hyperref[E]{\LE} and \hyperref[F]{\LF} converges to zero whenever
\begin{align*}
\sup_g & \, \mathbb{E} \, \mathbb{H}^2\big\{N\big(\widehat{M}_{n,2}^{1/2}G_n - \widehat{M}_{n,2}^{1/2} \widehat{G}_{n,1} + \widehat{G}_{n,1} , \sigma^2 I_{K_n^*} / (n-m)\big) \, , \, N\big(G_n , \sigma^2 I_{K_n^*} / (n-m)\big)\big\}  \\
& \, \leq \, 2 \, \sup_g \, \Big\{1 - \exp\Big(-\frac{n-m}{8\sigma^2} \, \mathbb{E} \big\|\big(\widehat{M}_{n,2}^{1/2} - I_{K_n^*}) (G_n - \widehat{G}_{n,1})\big\|^2\Big)\Big\}\,
\end{align*}
does so, where $\|\cdot\|$ stands for the Euclidean distance. Let $\widehat{\lambda}_{M,2,\ell}$ and $\widehat{m}_{2,\ell}$, $\ell=1,\ldots,K_n^*$, denote the eigenvalues and eigenvectors of the symmetric and positive semi-definite matrix $\widehat{M}_{n,2}$, respectively. Then, for all $v \in \mathbb{R}^{K_n^*}$,
\begin{align*} \big\|\big(\widehat{M}_{n,2}^{1/2} - I_{K_n^*}) v\big\|^2 & \, = \, \sum_{\ell=1}^{K_n^*} (v^\top \widehat{m}_{2,\ell})^2 \cdot \big(\widehat{\lambda}_{M,2,\ell}^{1/2} - 1\big)^2 \\ & \, \leq \, \sum_{\ell=1}^{K_n^*} (v^\top \widehat{m}_{2,\ell})^2 \cdot \big(\widehat{\lambda}_{M,2,\ell} - 1\big)^2 \, = \, \big\|\big(\widehat{M}_{n,2} - I_{K_n^*}) v\big\|^2\,.
\end{align*}
Therefore, it is sufficienct to ensure that the subsequent term (\ref{eq:transf1}) converges to zero as $n\to\infty$ uniformly with respect to $g$. Using independence of $\widehat{M}_{n,2}$ and $\widehat{G}_{n,1}$ as well as Lemma \ref{L:ON} in the last step,  we obtain
\begin{align} \nonumber
  (n-m)  \cdot  &\mathbb{E}  \big\|\big(\widehat{M}_{n,2} - I_{K_n^*}) (G_n - \widehat{G}_{n,1})\big\|^2 \\ \nonumber &  =  (n-m) \, \sum_{k=1}^{K_n^*} \mathbb{E} \Big| \frac1{n-m} \sum_{j=m+1}^n \psi_{n,k}(X_j) \sum_{k'=1}^{K_n^*} \langle g - \widehat{g}_{n,1} , \psi_{n,k'}\rangle_{p_X} \psi_{n,k'}(X_j)  \\ \nonumber & \hspace{9cm} -  \langle g-\widehat{g}_{n,1} , \psi_{n,k}\rangle_{p_X}\Big|^2 \\ \nonumber
 & \, = \, \sum_{k=1}^{K_n^*} \mathbb{E}\, \mbox{var}\Big(\psi_{n,k}(X_n) \sum_{k'=1}^{K_n^*} \langle g - \widehat{g}_{n,1} , \psi_{n,k'}\rangle_{p_X} \psi_{n,k'}(X_n) \, \big| \, \widehat{g}_{n,1}\Big) \\ \nonumber
& \, \leq \, \sum_{k=1}^{K_n^*} \mathbb{E}\, \psi_{n,k}^2(X_n) \Big|\sum_{k'=1}^{K_n^*} \langle g - \widehat{g}_{n,1} , \psi_{n,k'}\rangle_{p_X} \psi_{n,k'}(X_n)\Big|^2  \\
\nonumber & \, = \, \mathbb{E} \big|\big(g-\widehat{g}_{n,1}\big)^{[K_n^*]}(X_n)\big|^2\cdot \sum_{k=1}^ {K_n^*} \psi_{n,k}^2(X_n)   \\ \label{eq:transf1}
& \, \leq \, \rho^{-1} \cdot\big\{1 + K_n\cdot d_n \cdot (1+\pi^2/6)\big\} \cdot \mathbb{E}\|g - \widehat{g}_{n,1}\|_{p_X}^2\,.
\end{align}
The claim then follows with the definition of $K_n^*$.\end{proof}

In the experiment \hyperref[F]{\LF}, we then construct a pilot estimator $\widehat{G}_{n,2}$ of $G_n$ based on $\zeta_{n,2}$ and transform \hyperref[F]{\LF} into the experiment \hyperref[G]{\LG}\label{G}, where one observes $X_1,\ldots,X_n$, $\zeta_{n,2}$ and $$\zeta_{n,1}^* := Z_{n,1} - \widehat{M}_{n,1}^{1/2} \widehat{G}_{n,2} + \widehat{G}_{n,2}.
$$

Again, since $\widehat{G}_{n,2} - \widehat{M}_{n,1}^{1/2} \widehat{G}_{n,2}$ is measurable with respect to the $\sigma$-field generated by $X_1,\ldots,X_n$ and $\zeta_{n,2}$, we conclude equivalence of \hyperref[F]{\LF} and \hyperref[G]{\LG}. \\

In the experiment \hyperref[H]{\LH}\label{H}, one observes $X_1,\ldots,X_n$, $\zeta_{n,2}$ and $\zeta_{n,1}$ which is independent of $X_1,\ldots,X_n$ and $\zeta_{n,2}$ and has a Gaussian distribution with the expectation vector $G_n$ and the covariance matrix $\sigma^2 I_{K_n^*} / m$.

\smallskip
With the notation $\widehat{g}_{n,2} := \sum_{k=1}^{K_n^*} (\widehat{G}_{n,2})_k \cdot \psi_{n,k}$, we obtain
\begin{proposition}
    The experiments \hyperref[G]{\LG} and \hyperref[H]{\LH} are asymptotically equivalent if
    \begin{equation} \label{eq:GH}
\lim_{n\to\infty} \, K_n^* \cdot \sup_g \, \mathbb{E} \|g - \widehat{g}_{n,2}\|_{p_X}^2 \, = \, 0\, .
\end{equation}
\end{proposition}

The proof is analogous to the one of  Proposition \ref{prop: 2} and therefore omitted.

\subsection{Pilot estimators} \label{SS:pilot}

Now we define our pilot estimators $\widehat{G}_{n,1}$ and $\widehat{G}_{n,2}$ in experiments \hyperref[DI]{\LDI} and \hyperref[F]{\LF}, respectively. We impose that $K_n$ is an integer multiple of some integer parameter $J_n$; put $J_n^* := 1 + d_n(J_n-1)$. Whenever $K_n/J_n \to \infty$ this selection constraint can be satisfied. The linear space $\Phi_n^*$ equals $\Phi_n$ when $K_n$ is replaced by $J_n$. The corresponding orthonormal basis of $\Phi_n^*$ such that the inequality in Lemma \ref{L:ON} is satisfied (again, with $J_n$ instead of $K_n$) is called $\psi_{n,\ell}^*$, $\ell =1,\ldots,J_n^*$. Then,
\begin{align*}
\widehat{g}_{n,1} & \, := \, \sum_{\ell=1}^{J_n^*} \psi_{n,\ell}^* \, \sum_{k=1}^{K_n^*} \langle \psi_{n,\ell}^* , \psi_{n,k}\rangle_{p_X} \cdot \big(\widehat{M}_{n,1}^{1/2}Z_{n,1}\big)_k\,, \\
\widehat{g}_{n,2} & \, := \, \sum_{\ell=1}^{J_n^*} \psi_{n,\ell}^* \, \sum_{k=1}^{K_n^*} \langle \psi_{n,\ell}^* , \psi_{n,k}\rangle_{p_X} \cdot (\zeta_{n,2})_k\,,
\end{align*}
and $\widehat{G}_{n,q} := \{\langle \widehat{g}_{n,q} , \psi_{n,k}\rangle_{p_X}\}_{k=1,\ldots,K_n^*}$. The statistical risk of these estimators is bounded in the following lemma. Its proof is deferred to Section \ref{sec: 8}.
\begin{lem} \label{L:pilot}
It holds that
\begin{align*}
\sup_g \, \mathbb{E} \|\widehat{g}_{n,1} - g\|_{p_X}^2 & \, = \, {\cal O}\big(m^{-1} d_n^3 J_n + d_n J_n^{-2\beta}\big)\,, \\
\sup_g \, \mathbb{E} \|\widehat{g}_{n,2} - g\|_{p_X}^2 & \, = \, {\cal O}\big((n-m)^{-1} d_n J_n  + d_n J_n^{-2\beta}\big)\,,
\end{align*}
where constant factors may only depend on $\sigma$, $\rho$ and $C$ from (\ref{eq:Hoelder}).
\end{lem}
Using the optimal selection $J_n \asymp (n d_n^{-2})^{1/(2\beta+1)}$, we attain the convergence rate
$$ \sup_g \, \mathbb{E} \|\widehat{g}_{n,1} - g\|_{p_X}^2 \, + \, \sup_g \, \mathbb{E} \|\widehat{g}_{n,2} - g\|_{p_X}^2 \, = \, {\cal O}\big( d_n^{1 + 4\beta/(2\beta+1)} n^{-2\beta/(2\beta+1)}\big)\,. $$
Therefore, (\ref{eq:EF}) and (\ref{eq:GH}) are satisfied whenever
\begin{equation} \label{eq:Kn}
\lim_{n\to\infty} \, \big( K_n^{-1} \, n^{1/(2\beta+1)} d_n^{-2/(2\beta+1)} \, + \, K_n \cdot d_n^{2 + 4\beta/(2\beta+1)} \cdot n^{-2\beta/(2\beta+1)}\big) \, = \, 0\,.
\end{equation}
Consequently, together with the developments of Sections \ref{2} and \ref{3}, we have derived the following intermediate result.

\begin{proposition}
The original experiment  \hyperref[A]{\LA} is asymptotically equivalent to \hyperref[H]{\LH}
if the constraints (\ref{eq:AB}) and (\ref{eq:Kn}) are met.
\end{proposition}
\subsection{Removal of the covariates}

A closer inspection of the experiment \hyperref[H]{\LH} shows that it consists of three independent observations $(X_1,\ldots,X_n)$, $\zeta_{n,1}$ and $\zeta_{n,2}$, where the statistic $(X_1,\ldots,X_n)$ is conditionally ancillary, that is, its conditional distribution, given the other parts, does not depend on the parameter $\tilde{g}$, and hence can be omitted without losing any information on $\tilde{g}$. Furthermore, by the Neyman-Fisher lemma, we recognize $\zeta_n := (m/n) \zeta_{n,1} + (1-m/n) \zeta_{n,2}$ as a sufficient statistic for $\tilde{g}$. \\

Therefore, \hyperref[H]{\LH} is equivalent to the experiment \hyperref[I]{\LI}\label{I}, where only $\zeta_n \sim N(G_n, \sigma^2 I_{K_n^*}/n)$ is observed.

\section{White noise model} \label{FT}
%%%%%%%%%%%%%%%%%%%%%%%%%%%%%%%%%%%%%%%%%%%%%%%%%%%%%%%%%%%%%%%%

Our first goal is to replace the experiment \hyperref[I]{\LI} with the white noise experiment \hyperref[J]{\LJ}\label{J}, in which we observe the stochastic process $B_{1,n}$ on the domain $[0,1]^d$, which is defined by
$$ B_{1,n}(t) \, =  \, \int_0^{t_1} \cdots\int_0^{t_d} g^{[K_n^*]}(x) p_X^{1/2}(x)\, dx \, + \, \frac{\sigma}{\sqrt{n}}\, W(t)\,, \qquad t=(t_1,\ldots,t_d) \in [0,1]^d\,, $$
where $W$ denotes a $d$-variate Wiener process/Brownian sheet, that is,~$W$ is a stochastic process on the Borel subsets of $[0,1]^d$ such that $W(B_j)$, $j=1,\ldots,\ell$, are independent $N(0,\lambda_d(B_j))$-distributed random variables for all pairwise disjoint Borel subsets $B_1,\ldots,B_\ell$ of $[0,1]^d$ and any integer $\ell\geq 0$ where $\lambda_d$ stands for the $d$-dimensional Lebesgue-Borel measure; and we write $W(t) := W([0,t_1]\times \cdots\times[0,t_d])$.

\begin{proposition}
    The experiments \hyperref[I]{\LI} and \hyperref[J]{\LJ} are equivalent.
    \end{proposition}

\begin{proof}The action space of \hyperref[J]{\LJ} is the Banach space $C([0,1]^d)$ of all continuous functions on the domain $[0,1]^d$, equipped with the supremum norm and the corresponding Borel $\sigma$-field, hence the action space is Polish.
%For simple functions $f := \sum_{k=1}^m f_k \cdot {\bf 1}_{B_k}$ on the domain $[0,1]^d$ the integral $\int f(t) dW(t)$ is defined by $\sum_{k=1}^m f_k \cdot W(B_k)$. We write $L_2([0,1]^d)$ for $L_2(p_X)$ when $p_X$ is replaced by the uniform density on $[0,1]^d$. Any $f\in L_2([0,1]^d)$ is approximable by a sequence of simple functions with respect to the underlying $L_2([0,1]^d)$-norm. Then $\int f(t)  dW(t)$ is defined by the mean square limit of the integrals of the approximating simple functions. Note that $\int f(t) dW(t)$ is independent of the specific choice of the simple functions.
Note that $B_{1,n}$ can be reconstructed from its scores $\int \psi_{n,k}(x)\, p_X(x)^{1/2}dB_{1,n}(x)$, integer $k\geq 1$, when the definition of $\psi_{n,k}$ is extended for $k>K_n^*$ such that $\psi_{n,k}$, integer $k\geq 1$, are an orthonormal basis of $L_2(p_X)$, since
\begin{align*} B_{1,n}(t_1,\ldots,t_d) & \, = \, \idotsint {\bf 1}_{[0,t_1]\times \cdots\times [0,t_d]}(x) \, dB_{1,n}(x) \\ & \, = \, \sum_{k=1}^\infty \big\langle \psi_{n,k},{\bf 1}_{[0,t_1]\times \cdots\times [0,t_d]} \cdot p_X^{-1/2}\big\rangle_{p_X} \cdot \int \psi_{n,k}(x) p_X^{1/2}(x) \, dB_{1,n}(x)\,. \end{align*}
These integrals correspond to the Skorokhod integral, which coincides here with the Wiener–It\^{o} integral as the integrands are deterministic.
The above scores are independent and the distribution of the scores for $k>K_n^*$ does not depend on $\tilde{g}$, so that $\int \psi_{n,k}(x)\, p_X(x)^{1/2}dB_{1,n}(x)$, $k=1,\ldots,K_n^*$ form a sufficient statistic for $\tilde{g}$ in the experiment \hyperref[J]{\LJ}. Their joint distribution is equal to that of the random vector $\zeta_n$ in the experiment \hyperref[I]{\LI} from the previous section, which implies equivalence of experiments \hyperref[I]{\LI} and \hyperref[J]{\LJ}.
\end{proof}

Now we approximate \hyperref[J]{\LJ} by the following experiment \hyperref[K]{\LK}\label{K}, in which one observes the stochastic process $B_{2,n}$ defined by
$$ B_{2,n}(t) \, = \, \int_0^{t_1} \cdots\int_0^{t_d} g(x) p_X^{1/2}(x)\, dx \, + \, \frac{\sigma}{\sqrt{n}}\, W(t)\,, \qquad t=(t_1,\ldots,t_d) \in [0,1]^d\,. $$
\begin{proposition}
    The experiments \hyperref[J]{\LJ} and \hyperref[K]{\LK} are asymptotically equivalent  if condition (\ref{eq:AB}) is granted.
    \end{proposition}
\begin{proof}By the Cameron-Martin theorem, the Kullback-Leibler divergence $\mathbb{K}(\cdot,\cdot)$ between the probability measure of $B_{1,n}$ in the experiment \hyperref[J]{\LJ} and that of $B_{2,n}$ is equal to
$$ \mathbb{K}\big(\mathbb{P}_{B_{2,n}},\mathbb{P}_{B_{1,n}}\big) \, = \, \frac{n}{2\sigma^2} \cdot \big\|g - g^{[K_n^*]}\big\|_{p_X}^2\,.$$
Note that \hyperref[J]{\LJ} and \hyperref[K]{\LK} share the same action space. Thus, using (\ref{eq:2.2}) and the Bretagnolle-Huber inequality,  \hyperref[J]{\LJ} and \hyperref[K]{\LK} are asymptotically equivalent whenever (\ref{eq:AB}) is granted. \end{proof}

\begin{rem}\normalfont
By the experiment \hyperref[K]{\LK}, we arrive at a white noise model similar to \cite{R2008}. However, in the present note, asymptotic equivalence to the original regression experiment \hyperref[A]{\LA} has been proved for additive models under much less restrictive conditions with respect to the smoothness level $\beta$ (whereas $\beta>d/2$ is needed in \cite{R2008}), classes of non-necessarily periodic regression functions, potentially increasing $d=d_n$ and more general design densities $p_X$ (distinct from the uniform density).
\end{rem}

In view of the additive structure of $g$ in (\ref{eq:1.2}), we are not yet satisfied with \hyperref[K]{\LK} as our final target model. To proceed, we first need some notation. We introduce the Hilbert space $L_2([0,1],\mathbb{R}^d)$ which contains all measurable and square integrable functions $f$ from $[0,1]$ to $\mathbb{R}^d$ and is endowed with the inner product
$$ \langle f,\tilde{f} \rangle_{2,d} \, := \, \int_0^1 f(x)^\top \tilde{f}(x) \, dx\,, \qquad f,\tilde{f}\in L_2([0,1],\mathbb{R}^d)\,. $$
Subsequently, a crucial role is played by the linear operator $\Lambda$ from $L_2(p_X)$ to $L_2([0,1],\mathbb{R}^d)$ which maps any $f \in L_2(p_X)$ to the function
$$ t \, \mapsto \, \Big\{\idotsint f(x) p_X(x) \, dx_1\cdots dx_{\ell-1} dx_{\ell+1} \cdots dx_d\Big|_{x_\ell=t}\Big\}^\top_{\ell=1,\ldots,d}\,, \qquad t\in [0,1]\,. $$
Note that, by Jensen's inequality and Fubini's theorem,
$$ \|\Lambda f\|_{2,d}^2 \, \leq \, \sum_{\ell=1}^d \int f^2(x) p_X^2(x) dx \, \leq \, \rho^{-1} \cdot d \cdot \|f\|_{p_X}^2\,, \qquad \forall f \in L_2(p_X)\,, $$
so that $\Lambda$ is a bounded linear operator. Also, by Fubini's theorem and
\begin{align*}
\langle \Lambda f,\tilde{f} \rangle_{2,d} = \langle f, \Lambda^{\top} \tilde{f} \rangle_{p_X}, \quad f \in L_2(p_X), \tilde{f} \in L_2([0,1],\mathbb{R}^d),
\end{align*}
we deduce the following representation of its adjoint: the operator $\Lambda^\top$ maps any $f = (f_1,\ldots,f_d)^\top \in L_2([0,1],\mathbb{R}^d)$ to the function
$$ x=(x_1,\ldots,x_d) \, \mapsto \, \sum_{\ell=1}^d f_\ell(x_\ell)\,, \qquad x\in [0,1]^d\,, $$
which is located in $L_2(p_X)$.

\paragraph*{Sufficiency in \hyperref[K]{\LK}} Now we aim to extract a sufficient statistic in the experiment \hyperref[K]{\LK}. The Cameron-Martin theorem yields that the probability measure $\mathbb{P}_{B_{2,n}}$ of $B_{2,n}$ in the experiment \hyperref[K]{\LK} has the Radon-Nikodym density
\begin{equation} \label{eq:RNCM} \frac{d\mathbb{P}_{B_{2,n}}}{d\mathbb{P}_{W_{\sigma,n}}}\big(W_{\sigma,n}) \, = \,  \exp\Big( \frac{n}{\sigma^2} \, \int g(x) p_X^{1/2}(x) \, dW_{\sigma,n}(x)\Big) \cdot \exp\Big(-\frac{n}{2\sigma^2} \|g\|_{p_X}^2\Big)\,, \end{equation}
with respect to the probability measure $\mathbb{P}_{W_{\sigma,n}}$ of $W_{\sigma,n} := \sigma W/\sqrt{n}$. Recall $\tilde{g}= (g_1,\ldots,g_d)^\top$, and note that $\tilde{g} \in L_2([0,1],\mathbb{R}^d)$ and $\Lambda^\top \tilde{g} = g$. As $g$ is located in the range of $\Lambda^\top$, the general version of the Neyman-Fisher factorization lemma (see \cite{S85}, \cite{HS49}) and (\ref{eq:RNCM}) yield that the random linear functional $\Xi_n$ defined by
$$ \Xi_n \, : \, f \, \mapsto \, \int \big[\Lambda^\top f\big](x) \cdot p_X^{1/2}(x) \, dB_{2,n}(x)\,, \qquad f\in L_2([0,1],\mathbb{R}^d)\,, $$
forms a sufficient statistic for $\tilde{g}$ in the experiment \hyperref[K]{\LK} (where the observation $B_{2,n}$ from \hyperref[K]{\LK} is inserted in the statistic). \\

Therefore, \hyperref[K]{\LK} is equivalent to the experiment \hyperref[L]{\LL}\label{L}, in which only $\Xi_n$ is observed.

\paragraph*{Details on the experiment \hyperref[L]{\LL}} It follows that
\begin{align*}
\Xi_n(f) & \, = \,  \big\langle \Lambda^\top f , g\big\rangle_{p_X} \, + \, \frac{\sigma}{\sqrt{n}} \, \int \big[\Lambda^\top f\big](x) \cdot p_X^{1/2}(x) \, dW(x)\,.
\end{align*}
As concerns the scalar product, we have the identity
$$ \big\langle \Lambda^\top f , g\big\rangle_{p_X} \, = \, \big\langle \Lambda^\top f , \Lambda^\top \tilde{g}\big\rangle_{p_X} \, = \, \langle f , \Gamma \tilde{g}\rangle_{2,d}\,, $$
where $\Gamma := \Lambda \Lambda^\top$. This linear operator $\Gamma$ is bounded as the concatenation of two bounded linear operators, but not compact in general. Thus, for any deterministic $f\in L_2([0,1],\mathbb{R}^d)$, the random variable $\Xi_n(f)$ is normally distributed with the expectation $\langle f,\Gamma \tilde{g}\rangle_{2,d}$ and the variance
$$ \frac{\sigma^2}n \cdot \big\|\Lambda^\top f\big\|_{p_X}^2 \, = \, \frac{\sigma^2}n \cdot \langle f , \Gamma f\rangle_{2,d} \, = \, \frac{\sigma^2}n \cdot \big\|\Gamma^{1/2} f\big\|_{2,d}^2\,, $$
where $\Gamma^{1/2}$ denotes the unique square root of the self-adjoint and positive semi-definite bounded linear operator $\Gamma$ from the Hilbert space $L_2([0,1],\mathbb{R}^d)$ to itself (see \cite{SR1980}, Theorem VI.9, p.~196). Moreover we derive an alternative representation of $\Xi_n$ in the following: the random linear functional $\tilde{\Xi}_n$ is defined by
$$ \tilde{\Xi}_n \, : \, f \, \mapsto \, \int \big[\Gamma^{1/2} f\big](t) \, dR_n(t)\,, \qquad f \in L_2([0,1],\mathbb{R}^d)\,, $$
where $R_n$ denotes the Gaussian process
\begin{equation} \label{eq:Zn} R_n(t) \, = \, \int_0^t \big[\Gamma^{1/2}\tilde{g}\big](s) \, ds \, + \, \frac\sigma{\sqrt{n}} \cdot \tilde{W}(t)\,, \qquad t\in [0,1]\,, \end{equation}
with $\tilde{W}(t) := \big(W_1(t),\ldots,W_d(t)\big)^\top$, $t\in [0,1]$, consisting of independent standard Wiener processes $W_j$, $j=1,\ldots,d$, on the domain $[0,1]$, so that
$$ \tilde{\Xi}_n(f) \, = \, \langle f, \Gamma \tilde{g}\rangle_{2,d} \, + \, \frac{\sigma}{\sqrt{n}} \cdot \int \big[\Gamma^{1/2}f\big](t) \, d\tilde{W}(t)\,. $$
Hence, for any deterministic $f\in L_2([0,1],\mathbb{R}^d)$, the random variables $\tilde{\Xi}_n(f)$ and $\Xi_n(f)$ are identically distributed. Since both random functionals are linear, it follows that $\Xi_n$ and $\tilde{\Xi}_n$ share the same probability measure so that the random observation from experiment \hyperref[L]{\LL} may also be represented by $\tilde{\Xi}_n$.

\smallskip
Our target model is the experiment in which just $R_n$ in (\ref{eq:Zn}) is observed. To this aim some results from spectral theory are required. Thanks to the multiplication operator version of the spectral theorem for self-adjoint bounded linear operators (see e.g.~\cite{SR1980}, corollary of Theorem VII.3, p.~227), $\Gamma$ may be represented by $\Gamma = V^\top S_h V$ where $V$ denotes some unitary operator from $L_2([0,1],\mathbb{R}^d)$ onto the (separable) Hilbert space $L_2(\Omega,\mathfrak{A},\nu)$ of all measurable and square integrable real-valued functions on some $\sigma$-finite measure space $(\Omega,\mathfrak{A},\nu)$; and $S_h$ denotes the multiplication operator with some bounded and measurable function $h$ from $\Omega$ to $\mathbb{R}$. As $\Gamma$ is positive semi-definite, the function $h$ is nonnegative; and $\Gamma^{1/2} = V^\top S_{h^{1/2}} V$. We introduce the formal integral $W^\nu(f) = \int f dW^\nu$, $f\in L_2(\Omega,\mathfrak{A},\nu)$, as
$\sum_{j=1}^\infty \big\langle V^\top f , \xi_j\big\rangle_{2,d} \cdot \eta_j$,
for some fixed orthonormal basis $\{\xi_j\}_{j\geq 1}$ of $L_2([0,1],\mathbb{R}^d)$ and some sequence $(\eta_j)_j$ of i.i.d.~$N(0,1)$-distributed random variables; note that
$$\big\langle V^\top f , \xi_j\big\rangle_{2,d} = \big\langle f , V \xi_j\big\rangle_{\nu}\,, $$
where $\langle\cdot,\cdot,\rangle_\nu$ denotes the inner product of $L_2(\Omega,\mathfrak{A},\nu)$. We define the random linear functional $\tilde{\Xi}_n^\nu$ by
$$ \tilde{\Xi}_n^\nu \, : \, f \, \mapsto \, \big\langle f , h \cdot V\tilde{g}\big\rangle_\nu \, + \, \frac\sigma{\sqrt{n}} \, \int h^{1/2} \cdot f \, dW^\nu\,, \qquad f \in L_2(\Omega,\mathfrak{A},\nu)\,. $$

Let \hyperref[M]{\LM}\label{M} be the experiment in which the random linear functional $\tilde{\Xi}_n^\nu$ is observed.

\begin{proposition}
    The experiments \hyperref[L]{\LL} and \hyperref[M]{\LM} are equivalent.
    \end{proposition}

\begin{proof}For any deterministic $f \in L_2(\Omega,\mathfrak{A},\nu)$, the random variable $\tilde{\Xi}_n^\nu(f)$ is normally distributed with the expectation
$$ \big\langle f , h \cdot V\tilde{g}\big\rangle_\nu \, = \, \big\langle V^\top f , V^\top S_h V \tilde{g}\big\rangle_{2,d} \, = \, \big\langle V^\top f , \Gamma \tilde{g}\big\rangle_{2,d}\,, $$
and the variance
$$ \frac{\sigma^2}n  \sum_{j=1}^\infty \big\langle h^{1/2} f , V \xi_j\big\rangle_\nu^2 \, = \, \frac{\sigma^2}n \big\|V^\top\big(h^{1/2} f\big)\big\|_{2,d}^2 \, = \, \frac{\sigma^2}n  \big\| \Gamma^{1/2} V^\top f\big\|_{2,d}^2\,. $$
In the experiment \hyperref[L]{\LL}, we observe $\tilde{\Xi}_n$. The operator $\Gamma$ is known and so are $V$ and $h$; thus the random linear functional $f \mapsto \tilde{\Xi}_n\big(V^\top f\big)$, $f\in L_2(\Omega,\mathfrak{A},\nu)$, is empirically accessible. As shown above, this random functional and $\tilde{\Xi}_n^\nu$ are identically distributed. Again, since both functionals are linear, we may conclude from the coincidence of all finite-dimensional marginal distributions to identical distributions of the random functionals. Vice versa, when $\tilde{\Xi}_n^\nu$ is observed, the map $f \mapsto \tilde{\Xi}_n^\nu(Vf)$, $f \in L_2([0,1],\mathbb{R}^d)$, transforms this observation into that of the experiment \hyperref[L]{\LL}. \end{proof}

The next goal is to remove $h^{1/2}$ from the noise term. \\

Let \hyperref[N]{\LN}\label{N} be the experiment in which the random linear functional
\begin{align} \nonumber \label{eq:RLF2}
f \, \mapsto \, & \big\langle f , h^{1/2}\cdot V\tilde{g}\big\rangle_\nu \, + \, \frac\sigma{\sqrt{n}} \int  f \, dW^\nu \\
&  \, = \, \big\langle V^\top f , \Gamma^{1/2} \tilde{g}\big\rangle_{2,d} \, + \, \frac\sigma{\sqrt{n}} \int  f \, dW^\nu \,, \qquad f \in L_2(\Omega,\mathfrak{A},\nu)\, ,
\end{align}
is observed.

\begin{proposition}
The experiments \hyperref[M]{\LM} and \hyperref[N]{\LN} are equivalent.
\end{proposition}

\begin{proof}In the experiment \hyperref[M]{\LM}, we consider the random linear functional
\begin{equation} \label{eq:RLF1} f \, \mapsto \, \tilde{\Xi}_n^\nu\big(h^{-1/2} \cdot {\bf 1}_{\Omega\backslash N(h)} \cdot f\big) \, + \, \int {\bf 1}_{N(h)} \cdot f \, d\tilde{W}^\nu\,, \qquad f \in L_2(\Omega,\mathfrak{A},\nu)\,, \end{equation}
where $N(h) := h^{-1}(\{0\})$ and $\tilde{W}^\nu$ is an independent copy of $W^\nu$. Thus, it is generated independently of the observation. Note that the right-hand side of (\ref{eq:RLF1}) is equal to
\begin{equation*}
\big\langle f , h^{1/2}\cdot V\tilde{g}\big\rangle_\nu \, + \, \frac\sigma{\sqrt{n}} \int {\bf 1}_{\Omega\backslash N(h)} \cdot f \, dW^\nu \, + \, \frac\sigma{\sqrt{n}} \int {\bf 1}_{N(h)} \cdot f \, d\tilde{W}^\nu\,,
\end{equation*}
for any $f\in L_2(\Omega,\mathfrak{A},\nu)$. Above, we may replace $\tilde{W}^\nu$ by $W^\nu$ thanks to the independence of $\int a \,dW^\nu$ and $\int b \, dW^\nu$ for all orthogonal $a,b\in L_2(\Omega,\mathfrak{A},\nu)$. Therefore the map (\ref{eq:RLF1}) shares the same distribution with \eqref{eq:RLF2}, and
 we have shown that \hyperref[M]{\LM} may be transformed into the experiment \hyperref[N]{\LN}.
For the reverse transformation, just insert $h^{1/2}\cdot f$ instead of $f$ in (\ref{eq:RLF2}). Hence, equivalence of the experiments \hyperref[M]{\LM} and \hyperref[N]{\LN} follows. \end{proof}

Finally, let \hyperref[O]{\LO}\label{O} be the experiment in which the Gaussian process $R_n$ from (\ref{eq:Zn}) is observed.

\begin{proposition}
   The experiments \hyperref[N]{\LN} and \hyperref[O]{\LO} are equivalent.
    \end{proposition}

\begin{proof}Clearly,  \hyperref[O]{\LO} is equivalent to observing
$$ \tilde{R}_n \, :\, f \, \mapsto \, \int f \, dR_n \, = \, \big\langle f,\Gamma^{1/2} \tilde{g}\big\rangle_{2,d} \, + \, \frac\sigma{\sqrt{n}} \int f\, d\tilde{W}\,, \qquad f \in L_2([0,1],\mathbb{R}^d)\,. $$
Since $f\mapsto\int f\, dW^\nu$ and $f\mapsto\int V^\top f\, d\tilde{W}$ are identically distributed on the domain $f\in L_2(\Omega,\mathfrak{A},\nu)$ the experiments \hyperref[N]{\LN} and \hyperref[O]{\LO} are equivalent by just inserting $V^\top f$, $f\in L_2(\Omega,\mathfrak{A},\nu)$, instead of $f \in L_2([0,1],\mathbb{R}^d)$ in order to transform from \hyperref[O]{\LO} to \hyperref[N]{\LN}; and $V f$, $f\in L_2([0,1],\mathbb{R}^d)$, instead of $f \in L_2(\Omega,\mathfrak{A},\nu)$, for the other direction.\end{proof}

In summary, we have proved in this section  asymptotic equivalence of the experiment \hyperref[I]{\LI} from Section \ref{Loc} and the white noise experiment \hyperref[O]{\LO}, in which $R_n$ from (\ref{eq:Zn}) is observed, under the condition (\ref{eq:AB}).

\paragraph*{Conclusion} Combining all results so far we deduce that the original regression experiment \hyperref[A]{\LA} and the white noise model \hyperref[O]{\LO} are asymptotically equivalent under the conditions (\ref{eq:AB}) and (\ref{eq:Kn}). The smoothing parameter $K_n$ remains to be selected. We impose that $d=d_n \lesssim n^\alpha$ for some $\alpha\geq 0$; while, for $K_n$, we choose the approach $K_n \asymp n^\gamma$ for some $\gamma>0$. Choosing $\gamma$ such that
$$ (1+\alpha)/(2\beta) \, < \, \gamma \, < \, \big(2\beta - 2\alpha(4\beta+1)\big)/(2\beta+1)\,, $$
ensures validity of (\ref{eq:AB}) and (\ref{eq:Kn}). Then we are ready to state the following main result.
\begin{theo} \label{T:1}
Let the parameter space $\Theta$ consist of all functions
$$\tilde{g}(t) = \big(g_1(t),\ldots,g_d(t)\big)^\top,\quad t\in [0,1],
$$
whose components satisfy (\ref{eq:Hoelder}). The design density $p_X$ is assumed to be known and supported on $[0,1]^d$ with $\rho\leq p_X(t)\leq 1/\rho$ for all $t\in [0,1]^d$ and some fixed $\rho>0$. Let the dimension $d=d_n$ satisfy $d_n \lesssim n^\alpha$ for some $\alpha\geq 0$; and grant that
\begin{equation} \label{eq:T1} (2\beta+1)(\alpha+1) + 4\alpha\beta(4\beta+1) \, < \, 4\beta^2\,. \end{equation}
Then, the additive regression experiment \hyperref[A]{\LA} with $g$ as in (\ref{eq:1.2}) is asymptotically equivalent to the white noise experiment \hyperref[O]{\LO}, in which the stochastic process $R_n$ of (\ref{eq:Zn}) is observed.
\end{theo}

Let us discuss condition (\ref{eq:T1}) of Theorem \ref{T:1}, which can be rewritten as
\[ \beta > \frac{1+\sqrt{5}}{4}, \qquad
0 \leq \alpha < \frac{4\beta^2 - 2\beta - 1}{16\beta^2 + 6\beta + 1}\,. \]
For arbitrary but fixed dimension $d$, that is $\alpha=0$, only $\beta > (\sqrt5 + 1)/4$ is required to satisfy the constraints of the theorem. Importantly, $(\sqrt5 + 1)/4$ is a smoothness level smaller than~$1$. On the other hand, if we allow the dimension $d=d_n$ to increase slowly with respect to $n$, then the assumptions of Theorem \ref{T:1} are fulfilled, for instance, under the standard Lipschitz continuity condition ($\beta=1$) and $\alpha < 1/23$.

Next, we focus on the operator $\Gamma = \Lambda\Lambda^\top$ whose root occurs in the drift term of the stochastic process $R_n$ in (\ref{eq:Zn}) in the experiment \hyperref[O]{\LO}. We deduce that $\Gamma$ can be written as the operator-valued matrix $\{\Gamma_{j,k}\}_{j,k=1,\ldots,d_n}$ whose entries are the linear operators
$$ \Gamma_{j,k} f \, := \, \begin{cases} \int f(t) \, p_{X,k,j}(t,\bullet) dt\,, & \qquad \mbox{for }j\neq k\,, \\
p_{X,j}(\bullet) \cdot f(\bullet)\,, & \qquad \mbox{otherwise,} \end{cases} $$
which map from the Hilbert space $L_2([0,1],\mathbb{R})$ to itself, where $p_{X,k,j}$ and $p_{X,j}$ denote the marginal densities of the covariate components $(X_{1,k},X_{1,j})$ and $X_{1,j}$, respectively. If the marginal densities $p_{X,j}$ are the uniform density on $[0,1]$ and $p_X$ is a copula density, then the operators $\Gamma_{j,j}$ coincide with the identity operator on $L_2([0,1],\mathbb{R})$.

\section{Special cases and extensions} \label{SCE}
%%%%%%%%%%%%%%%%%%%%%%%%%%%%%%%%%%%%%%%%%%%%%%%%%%

In specific cases, the operator $\Gamma$ possesses a pleasant information-theoretic structure, which we aim to exploit.

\subsection{Pairwise independent components of the covariates} \label{SCE.1}
%%%%%%%%%%%%%%%%%%%%%%%%%%%%%%%%%%%%%%%%%%%%%%%%%%%%%%%%%%%%%%%%%%%%%%%%%%%%

Now we address the special case in which all components $X_{1,k}$, $k=1,\ldots,d_n$, of the design random variable $X_1$ are pairwise independent and, thus, the bivariate marginal densities admit the decomposition $p_{X,k,j}(x_k,x_j) = p_{X,k}(x_k)\cdot p_{X,j}(x_j)$. Note that, when to each of the components of the parameter $\tilde{g}\in \Theta$ some constant functions are added which add to zero, then regression function $g$ in (\ref{eq:1.2}) remains the same in the original experiment \hyperref[A]{\LA}. Therefore, without any loss of generality, the stochastic process $R_n$ in the asymptotically equivalent model \hyperref[O]{\LO} may be written as
$$ R_n(t) \, = \, \int_0^t \big[\Gamma^{1/2} g^*\big](s)\, ds \, + \, g_0 \cdot \int_0^t \big[\Gamma^{1/2} e_1\big](s)\, ds \, + \, \frac{\sigma}{\sqrt{n}} \cdot \tilde{W}(t)\,, \qquad t\in [0,1]\,, $$
with the centered component functions $g_k^* := g_k - \int_0^1 g_k(s) p_{X,k}(s) ds$, the $\mathbb{R}^{d_n}$-valued function $g^* := (g_1^*,\ldots,g_{d_n}^*)^\top$, the $d_n$-dimensional unit vector $e_1 := (1,0,\ldots,0)^\top$ and the shift
$$g_0 \, := \, \int g(x) p_X(x) dx \, = \, \sum_{k=1}^{d_n} \int_0^1 g_k(s) p_{X,k}(s) ds\,, $$
resulting in the canonical representation $g(x) = g_0 + \sum_{k=1}^{d_n} g_k^*(x_k)$; and $\tilde{g} = g^* + g_0\cdot e_1$. \\

On the other hand, consider the experiment \hyperref[Q]{\LQ}\label{Q} in which one observes $h_1^* \sim N(g_0,\sigma^2/n)$ and the independent Gaussian processes
\begin{equation} \label{eq:Rnk*} R_{n,k}^{**}(t) \, := \, \int_0^t p_{X,k}^{1/2}(s) \cdot g_k^*(s)\, ds \, + \, \frac{\sigma}{\sqrt{n}} \cdot W_k^*(t)\,, \qquad t\in [0,1], \, k=1,\ldots,d_n\,, \end{equation}
with the components $W_k^*$ of $W^*$.

\begin{theo} \label{T:2}
Grant the conditions of Theorem \ref{T:1} and assume that all components $X_{1,k}$, $k=1,\ldots,d_n$, of the random design variable $X_1$ are pairwise independent. Then the original additive regression experiment \hyperref[A]{\LA} is asymptotically equivalent to the experiment \hyperref[Q]{\LQ}, in which the independent Gaussian processes $R_{n,k}^{**}$ in   (\ref{eq:Rnk*}) and -- independently of these -- a $N(g_0,\sigma^2/n)$-distributed random variable are observed.
\end{theo}

\begin{proof}
Writing $h_1 := \Gamma^{1/2} e_1$ we deduce that
$$ \|h_1\|_{2,d}^2 \, = \, \langle e_1 , \Gamma e_1\rangle_{2,d} \, = \, \int_0^1 p_{X,1}(s) ds \, = \, 1\,,$$
and that $\Gamma$ -- restricted to such centered functions -- reduces to a componentwise multiplication operator
\begin{equation} \label{eq:Gammaw}\Gamma g^* \, = \, (p_{X,1}\cdot g_1^*,\ldots,p_{X,d_n} \cdot g_{d_n}^*)^\top\,, \end{equation} as well as
$$ \langle h_1 , \Gamma^{1/2} g^*\rangle_{2,d} \, = \, \langle e_1 , \Gamma g^*\rangle_{2,d} \, = \, \int_0^1 p_{X,1}(s) \, g_1^*(s) \, ds \, = \, 0\,.$$
There exist functions $h_j \in L_2([0,1],\mathbb{R}^d)$, integer $j\geq 2$, such that the $h_j$, integer $j\geq 1$, form an orthonormal basis of $L_2([0,1],\mathbb{R}^d)$. Clearly, the experiment \hyperref[O]{\LO} is equivalent to observing all scores $\widehat{h}_j := \int_0^1 h_j(t)^\top \, dR_n(t)$, $j\geq 1$, where
\begin{align*}
\widehat{h}_1 & \, = \, g_0 \, + \, \frac{\sigma}{\sqrt{n}} \cdot \int_0^1 h_1(t)^\top \, d\tilde{W}(t)\,, \\
\widehat{h}_j & \, = \, \int_0^1 h_j(t)^\top \, \big[\Gamma^{1/2}g^*](t) \, dt \, + \, \frac{\sigma}{\sqrt{n}} \cdot \int_0^1 h_j(t)^\top \, d\tilde{W}(t)\,, \qquad j\geq 2\,,
\end{align*}
so that all these scores are independent. Hence the random sequences $(\widehat{h}_j)_{j\geq 1}$ and $(h^*_j)_{j\geq 1}$ are identically distributed where $h^*_1 := \widehat{h}_1$ and
$$ h^*_j \, := \, \int_0^1 h_j(t)^\top dR_n^*(t)\,, \qquad j\geq 2\,,$$
with
$$ R_n^*(t) \, := \, \int_0^t \big[\Gamma^{1/2}g^*](s) \, ds \, + \, \frac{\sigma}{\sqrt{n}} \cdot W^*(t)\,, \qquad t\in [0,1]\,,$$
for an independent copy $W^*$ of the stochastic process $\tilde{W}$. \\

Consider that
$$ \tilde{h}_1 \, := \, \int_0^1 h_1(t)^\top dR_n^*(t) \, = \, \frac{\sigma}{\sqrt{n}} \cdot \int_0^1 h_1(t)^\top dW^*(t)\,,$$
from what follows conditional ancillarity of the statistic $\tilde{h}_1$ given $h^*_j$, $j\geq 1$, in the statistical experiment \hyperref[P]{\LP}\label{P}, in which one observes $h_1^*$ and the stochastic process $R_n^*$ on its domain $[0,1]$. \\

We have shown equivalence of the experiments \hyperref[O]{\LO} and \hyperref[P]{\LP}. By the Cameron-Martin theorem, the probability measure $\mathbb{P}_{R_n^*}$ of $R_n^*$ has the Radon-Nikodym density
$$ \frac{d\mathbb{P}_{R_n^*}}{d\mathbb{P}_{W^*_{\sigma,n}}}(W^*_{\sigma,n}) \, = \, \exp\Big( \frac{n}{\sigma^2} \int_0^1 \big[\Gamma^{1/2} g^*\big]^\top(s) \, dW^*_{\sigma,n}(s)\Big) \cdot \exp\Big(- \frac{n}{2\sigma^2} \, \|\Gamma^{1/2} g^*\|_{2,d}^2\Big)\,,$$
with respect to the probability measure $\mathbb{P}_{W^*_{\sigma,n}}$ of $W^*_{\sigma,n} := \sigma W^*/\sqrt{n}$ where $g^*$ is located in the closed linear subspace
\begin{align*} & L_2^0([0,1],\mathbb{R}^d) \\ & \hspace{1cm} := \Big\{f=(f_1,\ldots,f_d)^\top \in L_2([0,1],\mathbb{R}^d)  :  \int_0^1 f_k(t) p_{X,k}(t) dt \, = \, 0\,, \forall k=1,\ldots,d\Big\}, \end{align*}
of $L_2([0,1],\mathbb{R}^d)$, which is therefore a Hilbert space itself. Hence, the Neyman-Fisher lemma yields sufficiency of the statistic which consists of $h_1^*$ and the random linear functional
\begin{align} \nonumber
f \, \mapsto \, & \int_0^1 \big[\Gamma^{1/2} f]^\top(t) \, dR_n^*(t) \\ \label{eq:Gammaw1} & \, = \, \langle f , \Gamma g^*\rangle_{2,d} \, + \, \frac{\sigma}{\sqrt{n}} \int_0^1 \big[\Gamma^{1/2} f\big]^\top(t) \, dW^*(t)\,,
\end{align}
on the domain $L_2^0([0,1],\mathbb{R}^d)$ in the experiment \hyperref[P]{\LP}. Note that
\begin{equation} \label{eq:L0f}
\Gamma f \, = \, \big(p_{X,1} \cdot f_1 , \ldots , p_{X,d}\cdot f_{d}\big)^\top\,,
\end{equation}
for all $f = (f_1,\ldots,f_{d})^\top \in L_2^0([0,1],\mathbb{R}^d)$.

Writing $R_n^{**} := (R_{n,1}^{**}, \ldots, R_{n,d}^{**})^\top$, we use the Cameron-Martin theorem and the Neyman-Fisher lemma again to show that $h_1^*$ and the random linear functional
\begin{align} \nonumber f = (f_1,\ldots,f_d)^\top  \, \mapsto \, & \int_0^1 \big(p_{X,k}^{1/2}(t) \cdot f_k(t)\big)_{k}^\top\, dR_n^{**}(t) \\ \label{eq:Gammaw2}
& \, = \, \big\langle f , \big(p_{X,k}\cdot g_k^*\big)_{k}\big\rangle_{2,d} \, + \, \frac{\sigma}{\sqrt{n}} \, \int_0^1 \big(p_{X,k}^{1/2}(t)\cdot f_k(t)\big)_{k}^\top \, dW^*(t)\,,
\end{align}
on the domain $L_2^0([0,1],\mathbb{R}^d)$, form a sufficient statistic in the experiment \hyperref[Q]{\LQ}. We realize that, when any joint deterministic $f\in L_2^0([0,1],\mathbb{R}^d)$ is inserted in the linear functionals (\ref{eq:Gammaw1}) and (\ref{eq:Gammaw2}), the corresponding random variables share the same distribution $N\big(\langle f,\Gamma g^*\rangle_{2,d} \, , \, \sigma^2 \langle f,\Gamma f\rangle_{2,d}/n\big)$ thanks to (\ref{eq:Gammaw}) and (\ref{eq:L0f}). The linearity of both functionals yields that random functionals are identically distributed. This implies that the experiments \hyperref[P]{\LP} and \hyperref[Q]{\LQ} are equivalent.
\end{proof}

Further asymptotic equivalence can be established if we fix that the dimension $d=d_n$ is bounded. The results of \cite{BL1996} yield that the observation of the process $R_{n,k}^{**}$ becomes asymptotically equivalent to the univariate Gaussian regression model in which one observes $n$  i.i.d.~regression data under the design density $p_{X,k}$, the regression function $g_k^*$ and the error variance $\sigma^2$. Combining this with Theorem \ref{T:2} we obtain

\begin{cor} \label{T:3}
Grant the assumptions of Theorem \ref{T:1} and \ref{T:2}. In addition, impose $\sup_n d_n < \infty$. Then the original additive regression model \hyperref[A]{\LA} is asymptotically equivalent to the statistical experiment \hyperref[R]{\LR}\label{R} in which independent univariate regression data $(X_{j,k}^*,Y_{j,k}^*)$, $j=1,\ldots,n$, $k=1,\ldots,d_n$, and -- independently of these -- a $N(g_0,\sigma^2/n)$-distributed random variable are observed where $X_{j,k}^*$ has the density $p_{X,k}$ and $$Y_{j,k}^*\mid X_{j,k}^* \sim N\big(g_k^*(X_{j,k}^*),\sigma^2\big)\,.$$
\end{cor}

\subsection{Further approximation in semiparametric submodels} \label{SCE.2}
%%%%%%%%%%%%%%%%%%%%%%%%%%%%%%%%%%%%%%%%%%%%%%%%%%%%%%%%%%%%%%%%%%%%%%%%%%%%

We turn to the general setting of the experiment \hyperref[O]{\LO} again, where the components of the covariates are not necessarily pairwise independent. Moreover, we consider the dimension $d$ as bounded. By $\{\xi_\ell\}_{l\geq 1}$ we fix an arbitrary orthonormal basis of the Hilbert space $L_2([0,1],\mathbb{R}^d)$. In this subsection we impose that the true function $\tilde{g}$ lies in
$$ \Theta_L := \Theta \cap \mbox{span}(\xi_\ell : \ell \in L)\,, $$
where $L=L_n$ is some finite subset of $\mathbb{N}$ which may depend on $n$. We define
\begin{equation} \label{eq:Snj} S_{n,j}^*(t) \, := \, \int_0^t p_{X,j}^{1/2} (s) \cdot g_j(s) \, ds \, + \, \frac{\sigma}{\sqrt{n}} \, W_j(t)\,, \quad t\in [0,1]\,, \end{equation}
and $\tilde{g} \in \Theta_L$.

\begin{theo} \label{T:SCE.1}
Consider the parameter space $\Theta_L$ which is a subset of $\Theta$ from Theorem \ref{T:1} and of the linear hull of the $\xi_\ell$, $\ell\in L$, for a fixed but arbitrary orthonormal basis $(\xi_\ell)_{\ell\geq 1}$ of $L_2([0,1],\mathbb{R}^d)$ and a finite set $L=L_n$ containing only positive integers. Assume that $\min L \to\infty$ as $n$ tends to $\infty$. Grant the conditions of Theorem \ref{T:1}; and that the dimension $d$ does not depend on $n$. Then the experiment \hyperref[A]{\LA} is asymptotically equivalent to the experiment \hyperref[S]{\LS}\label{S}, in which the process $S_n^* := (S_{n,j}^*)_{j=1,\ldots,d}$ from (\ref{eq:Snj}) is observed.
\end{theo}

\begin{proof}
Again, we use the Cameron-Martin theorem to derive
\begin{align*} \frac{d\mathbb{P}_{R_n}}{d \mathbb{P}_{W_{\sigma,n}}}(W_{\sigma,n})  \, = \, & \exp\Big(\frac{n}{\sigma^2} \sum_{\ell\in L} \langle \tilde{g} , \xi_\ell\rangle_{2,d} \cdot \int_0^1 \big[\Gamma^{1/2} \xi_\ell \big]^\top(s) \, dW_{\sigma,n}(s)\Big) \\ & \cdot \exp\Big(-\frac{n}{2\sigma^2} \big\|\Gamma^{1/2} \tilde{g}\big\|_{2,d}^2\Big)\,, \end{align*}
and the Neyman-Fisher factorization lemma to show that the scores $\int_0^1 [\Gamma^{1/2}\xi_\ell]^\top(s) dR_n(s)$, $\ell\in L$, are a sufficient statistic in the experiment \hyperref[O]{\LO} under these constraints. These scores form the random vector $R'_n \sim N\big(\Gamma^{[L]}\theta_L \, , \, \sigma^2 \Gamma^{[L]}/n\big)$ with the parameter vector $\theta_L := (\langle \xi_\ell , \tilde{g}\rangle_{2,d})_{\ell\in L}$ and the known matrix
\begin{equation} \label{eq:GammaL} \Gamma^{[L]} := (\langle \xi_\ell\, , \, \Gamma \xi_{\ell'}\rangle_{2,d})_{\ell,\ell' \in L}. \end{equation}
The operator $\Gamma$ can be decomposed by $\Gamma = \Gamma_M + \Gamma_{HS}$ with the componentwise multiplication operator
$$ \Gamma_M \, : \, \tilde{f} = (f_1,\ldots,f_d)^\top \, \mapsto \, (p_{X,1}\cdot f_1,\ldots,p_{X,d}\cdot f_d)^\top\,, $$
and the off-diagonal operator
$$ \Gamma_{HS} \, : \, \tilde{f} = (f_1,\ldots,f_d)^\top \, \mapsto \, \Big( \sum_{k\neq j}^d \int_0^1 f_k(t) \, p_{X,k,j}(t,\bullet)\, dt\Big)_{j=1,\ldots,d}^\top\,. $$
Both $\Gamma_M$ and $\Gamma_{HS}$ are self-adjoint linear operators from $L_2([0,1],\mathbb{R}^d)$ to itself where $\Gamma_M$ is a positive-definite bijection which satisfies the ellipticity condition
\begin{equation} \label{eq:ellip} \langle \tilde{f} , \Gamma_M \tilde{f}\rangle_{2,d} \, \geq \, \rho \cdot \|\tilde{f}\|_{2,d}^2 \, , \qquad \forall \tilde{f} \in L_2([0,1],\mathbb{R}^d)\,, \end{equation}
while $\Gamma_{HS}$ is a Hilbert-Schmidt operator with the squared Hilbert-Schmidt norm
\begin{align} \nonumber  \|\Gamma_{HS}\|_{HS}^2 & \, = \, \sum_{\ell\geq 1} \|\Gamma_{HS}\xi_\ell\|_{2,d}^2 \, = \, \sum_{j=1}^d \int_0^1 \sum_{\ell\geq 1} \Big(\sum_{k=1}^d \int_0^1 \xi_{\ell,k}(t) p_{X,k,j}(t,s) dt\Big)^2 ds \\ \label{eq:SCE.2.1} & \, = \, \sum_{j\neq k}^d \iint p_{X,k,j}^2(t,s) \, dt \, ds \, \leq \, d(d-1)/\rho^2\,, \end{align}
with $\xi_\ell = (\xi_{\ell,k})^\top_{k=1,\ldots,d}$ where we put $p_{X,k,k} :\equiv 0$. We introduce the positive definite matrix $\Gamma_M^{[L]} := \{\langle \xi_\ell \, , \, \Gamma_M \xi_{\ell'}\rangle_{2,d}\}_{\ell,\ell'\in L}$, where the spectral norm of its  inverse is bounded from above by $1/\rho$ thanks to (\ref{eq:ellip}). The squared Frobenius distance between the matrices $\Gamma^{[L]}$ and $\Gamma_M^{[L]}$ is bounded from above as follows,
$$ \big\| \Gamma^{[L]} - \Gamma_M^{[L]}\big\|_{F}^2 \, = \, \sum_{\ell,\ell'\in L} \langle \xi_{\ell'} , \Gamma_{HS}\xi_\ell\rangle_{2,d}^2 \, \leq \,\sum_{\ell\in L} \|\Gamma_{HS}\xi_l\|_{2,d}^2\,,  $$
where the right side of the inequality converges to zero whenever $\min L$ tends to infinity as $n\to \infty$ thanks to (\ref{eq:SCE.2.1}). This suffices to guarantee invertibility of the matrix $\Gamma^{[L]}$ for $n$ sufficiently large. Thus observing $R'_n$ is asymptotically equivalent to observing $R''_n \sim N\big(\theta_L , \sigma^2 (\Gamma^{[L]})^{-1} / n\big)$. Furthermore,
$$ \big\| \big(\Gamma_M^{[L]}\big)^{-1/2} \big(\Gamma^{[L]} - \Gamma_M^{[L]}\big) \big(\Gamma_M^{[L]}\big)^{-1/2} \big\|_{F}^2 \, \leq \, \rho^{-2} \cdot \big\| \Gamma^{[L]} - \Gamma_M^{[L]}\big\|_{F}^2\,, $$
converges to zero as $\min L \to \infty$. By the inequality (A.4) in \cite{R2011} and the fact that $\|I_L - C_n^{-1}\|_F$ converges to zero whenever $\|I_L - C_n\|_F$ does so (for any sequence $(C_n)_n$ of invertible $L\times L$-matrices), the experiment \hyperref[O]{\LO} -- under the restricted parameter space $\Theta_L$ -- is asymptotically equivalent to the experiment in which the random variable $S''_n \sim N\big( \theta_L \, , \, \sigma^2(\Gamma_M^{[L]})^{-1}/n\big)$ is observed.

Repeating the arguments from the beginning of this proof in reverse order with $\Gamma$ being replaced by $\Gamma_M$, the statistic $S_n' \sim N\big(\Gamma_M^{[L]} \theta_L , \sigma^2 \Gamma_M^{[L]}/n\big)$ represents a sufficient statistic in the experiment \hyperref[S]{\LS} in which one observes the $\mathbb{R}^d$-valued stochastic process $S_n^* = (S^*_{n,j})_{j=1,\ldots,d}$ on the domain $[0,1]$.
\end{proof}

Experiment \hyperref[S]{\LS} is significantly easier to handle than experiment \hyperref[O]{\LO} as we get rid of the off-diagonal operator $\Gamma_{HS}$. Note that $\min L$ may tend to infinity arbitrarily slowly so that Theorem \ref{T:SCE.1} is useful to deduce asymptotic minimax lower bounds in plenty of settings of nonparametric curve estimation. This also explains that sharp asymptotic results from nonparametric univariate regression are attainable for any component function in  \cite{HKM2006}. On the other hand, for the estimation of functionals of $\tilde{g}$ when parametric rates are achievable, one should consider the experiment \hyperref[O]{\LO}, in which the off-diagonal operator is involved.

\begin{rem}[Unknown design density]\normalfont\label{rem: 6.4}
Finally, since all results have been established under the assumption of known design density $p_X$, we will now address the question whether the sufficient statistic $S_n''$ in the experiment \hyperref[S]{\LS} can also be asymptotically obtained in the original regression experiment \hyperref[A]{\LA} under unknown $p_X$ and the conditions of Theorem \ref{T:SCE.1}. In the model \hyperref[A]{\LA} we construct the statistic
$$ \widehat{S}_n \, := \, \Big(\frac1n \sum_{j=1}^n Y_j \sum_{k=1}^d \xi_{\ell, k}(X_{j,k})\Big)_{\ell\in L}\,, $$
with $X_j = (X_{j,1},\ldots,X_{j,d})^\top$ where $\widehat{S}_n$ has the conditional distribution  $N\big(\widehat{\Gamma}^{[L]} \theta_L \, , \, \sigma^2 \widehat{\Gamma}^{[L]}/n\big)$ given $X_1,\ldots,X_n$ and
$$\widehat{\Gamma}^{[L]} = \Big\{\frac1n \sum_{j=1}^n \sum_{k,k'=1}^d \xi_{\ell,k}(X_{j,k}) \xi_{\ell',k'}(X_{j,k'})\Big\}_{\ell,\ell'\in L}\,. $$
We derive $\mathbb{E} \widehat{\Gamma}^{[L]} = \Gamma^{[L]}$ (see \eqref{eq:GammaL}) and that
\begin{align} \nonumber \mathbb{E} \big\|\widehat{\Gamma}^{[L]} - \Gamma^{[L]}\big\|_F^2 & \, \leq \, \frac1n \sum_{\ell,\ell'\in L} \mathbb{E}\Big(\sum_{k,k'=1}^d \xi_{\ell,k}(X_{1,k}) \xi_{\ell',k'}(X_{1,k'})\Big)^2 \\ \nonumber & \, \leq \, \frac1n \sum_{\ell,\ell'\in L} d^2 \sum_{k=1}^d \|\xi_{\ell,k}\|_\infty^2 \cdot \mathbb{E} \sum_{k'=1}^d \xi_{\ell',k'}^2(X_{1,k'}) \\ \label{eq:upper} & \, \leq \, \frac1n \cdot (\# L)^2 \cdot d^2\cdot \rho^{-1} \, \max_{\ell\in L} \sum_{k=1}^d \|\xi_{\ell,k}\|_\infty^2\,. \end{align}
Thus, if the conditions of Theorem \ref{T:SCE.1} are granted and the right side of (\ref{eq:upper}) converges to $0$, then we can use the previous arguments (including (A.4) in \cite{R2011}) to show that the total variation distance between the distribution of $(\widehat{\Gamma}^{[L]})^{-1}\widehat{S}_n$ and that of $S''_n$ in the experiment \hyperref[S]{\LS} under known $p_X$ vanishes asymptotically uniformly with respect to $\tilde{g}\in \Theta_L$.\end{rem} %\newpage

\section{Proofs of the side results}\label{sec: 8}
%%%%%%%%%%%%%%%%%%%%%%%%%%%%%%%%%%%%

\begin{proof}[Proof of Lemma \ref{L:ON}]
We first verify linear independence of the functions $\varphi_{n,0}$ and $\varphi_{n,k,k'}$, $k=1,\ldots,K_n-1$, $k'=1,\ldots,d$. Assume that $\mu_0 + \sum_{k,k'=1}^{K_n-1,d} \mu_{k,k'} \cdot \varphi_{n,k,k'} \equiv 0$ for some real-valued coefficients $\mu_0$ and $\mu_{k,k'}$; then
\begin{align*}
 0  &= \Big\|\mu_0 + \sum_{k,k'=1}^{K_n-1,d} \mu_{k,k'} \cdot \varphi_{n,k,k'}\Big\|^2_{p_X} \\ & \, \geq \, \rho \Big(\mu_0^2 + 2 \mu_0 \sum_{k,k'=1}^{K_n-1,d} \mu_{k,k'} / K_n + \sum_{k,k'=1}^{K_n-1,d} \mu_{k,k'}^2 / K_n + \sum_{k'_1\neq k'_2}^d \sum_{k_1,k_2=1}^{K_n-1} \mu_{k_1,k'_1} \mu_{k_2,k'_2} /K_n^2\Big) \\
&  =  \rho \Big\{\Big(\mu_0 + \sum_{k,k'=1}^{K_n-1,d} \mu_{k,k'} / K_n\Big)^2 + \sum_{k,k'=1}^{K_n-1,d} \mu_{k,k'}^2 / K_n - \sum_{k'=1}^d \sum_{k_1,k_2=1}^{K_n-1} \mu_{k_1,k'} \mu_{k_2,k'} /K_n^2\Big\} \\
&  =  \rho \Big\{\Big(\mu_0 + \sum_{k,k'=1}^{K_n-1,d} \mu_{k,k'} / K_n\Big)^2 + \sum_{k,k'=1}^{K_n-1,d} \mu_{k,k'}^2 / K_n^2  \\ & \hspace{7.3cm} + \sum_{k'=1}^d \sum_{k_1,k_2=1}^{K_n-1} (\mu_{k_1,k'} - \mu_{k_2,k'})^2 / (2K_n^2)\Big\}\,,
\end{align*}
so that $\mu_{k,k'}=0$ for all $k=1,\ldots,K_n-1$, $k'=1,\ldots,d$, and consequently $\mu_0=0$.\\

We can now start the construction of the basis by choosing $\psi_{n,1} := \varphi_{n,0} \equiv 1$ as an element of $\Phi_n$. Then, for any $k=1,\ldots,K_n-1$ and $k'=1,\ldots,d$,  we define
\begin{align*} \tilde{\psi}_{n,k,k'}(x) := &  K_n^{-1/2} \cdot (1+1/k)^{-1/2} \cdot \Big(\frac{K_n}k \cdot{\bf 1}_{[0,k/K_n)}(x_{k'}) \, - \, K_n\cdot {\bf 1}_{[k/K_n,(k+1)/K_n)}(x_{k'})\Big)\,, \end{align*}
on the domain $x=(x_1,\ldots,x_d)\in [0,1]^d$. Clearly any $\tilde{\psi}_{n,k,k'}$ is located in the linear hull of $\varphi_{n,\kappa,k'}$, $\kappa=0,\ldots,k$, and  thus in $\Phi_n$. By $\tilde{\psi}_{n,j}$, $j=2,\ldots,K_n^*$, we denote some enumeration of $\tilde{\psi}_{n,k,k'}$, $k=1,\ldots,K_n-1$, $k'=1,\ldots,d$. Consider the Gram matrices
\begin{align*}
\mbox{Gr}_n \, &:= \, \big\{\langle \tilde{\psi}_{n,j} \, , \, \tilde{\psi}_{n,j'}\rangle_{p_X}\big\}_{j,j'=1,\ldots,K_n^*}\,,\\
\tilde{\mbox{Gr}}_n \, &:= \, \Big\{\idotsint \tilde{\psi}_{n,j}(x)\, \tilde{\psi}_{n,j'}(x) \, dx\Big\}_{j,j'=1,\ldots,K_n^*}\,,
\end{align*}
where $\tilde{\mbox{Gr}}_n$ turns out to be the $K_n^*\times K_n^*$-identity matrix $I_{K_n^*}$. Recall that a Gram matrix is positive semi-definite, that is, all eigenvalues are nonnegative. Since $p_X \geq \rho$, it follows that the Gram matrix ${\mbox{Gr}}_n -\tilde{\mbox{Gr}}_n$ is positive semi-definite, which implies that the smallest eigenvalue $\underline{d}_{n}$ of $\mbox{Gr}_n$ is bounded from below by $\rho$.

Let $U_n \mbox{Gr}_n U_n^{\top}$ be the spectral decomposition of $\mbox{Gr}_n$, with orthogonal matrix $U_n = (U_{n,k,j})_{k,j=1,\ldots,K_n^*}$ and diagonal matrix $D_n$ with eigenvalues $d_{n,k}$, $k=1,\ldots,K_n^*$. As an immediate consequence, the functions
$$ \psi_{n,j} \, := \, d_{n,j}^{-1/2} \, \sum_{k=1}^{K_n^*} U^\top_{n,k,j} \cdot \tilde{\psi}_{n,k}\,, \qquad j=1,\ldots,K_n^*\,,$$
form an orthonormal basis of $\Phi_n$ with respect to the inner product $\langle \cdot,\cdot \rangle_{p_X}$. Moreover, as $d_{n,j} \geq \underline{d}_n \geq \rho$ and using also the orthogonality of $U_n$, it follows that
\begin{align*}
\sum_{j=1}^{K_n^*} \psi_{n,j}^2(x) & \, \leq \, \rho^{-1} \, \sum_{j=1}^{K_n^*} \Big(\sum_{k=1}^{K_n^*} U^\top_{n,k,j} \cdot \tilde{\psi}_{n,k}(x)\Big)^2 \, = \, \rho^{-1} \, \sum_{k=1}^{K_n^*} \tilde{\psi}_{n,k}^2(x) \\ & \, = \, \rho^{-1} \cdot \Big(1 + \sum_{k',k=1}^{d,K_n-1} \tilde{\psi}_{n,k,k'}^2(x)\Big) \,, \qquad \forall x\in [0,1]^d\,.
\end{align*}
However, since
\begin{align*}
 \sum_{k',k=1}^{d,K_n-1} & \tilde{\psi}_{n,k,k'}^2(x) \\ & \, \leq \, K_n^{-1}  \sum_{k'=1}^d \Big(K_n^2  \sum_{k=1}^{K_n-1} k^{-2}\cdot {\bf 1}_{[0,k/K_n)}(x_{k'})  +  K_n^2 \sum_{k=1}^{K_n-1} {\bf 1}_{[k/K_n,(k+1)/K_n))}(x_{k'})\Big) \\
& \, \leq \, K_n \cdot d \cdot (1 + \pi^2/6)\,
\end{align*}
for all $x\in [0,1]^d$, this completes the proof of the lemma.
\end{proof}

\begin{proof}[Proof of Lemma \ref{L:pilot}] By Parseval's identity, consider that
\begin{align} \nonumber
\mathbb{E} & \|\widehat{g}_{n,1} - g\|_{p_X}^2 \\ \nonumber & = \sum_{\ell=1}^{J_n^*} \mathbb{E} \Big|\frac1m \sum_{j=1}^m \sum_{k,k'=1}^{K_n^*} \langle \psi_{n,\ell}^* , \psi_{n,k}\rangle_{p_X} \psi_{n,k}(X_j) \psi_{n,k'}(X_j) \, \langle g , \psi_{n,k'}\rangle_{p_X} - \langle g,\psi_{n,\ell}^*\rangle_{p_X}\Big|^2 \\ \nonumber & \quad + \, \frac{\sigma^2}m \,\sum_{\ell=1}^{J_n^*}  \mathbb{E} \Big|\sum_{k,k'=1}^{K_n^*} \langle \psi_{n,\ell}^* , \psi_{n,k}\rangle_{p_X} (\widehat{M}_{n,1})^{1/2}_{k,k'} \, \eta_{n,k'}\Big|^2 \\ \label{eq:risk1}
& \quad + \, \sum_{\ell > J_n^*} \langle g , \psi_{n,\ell}^*\rangle_{p_X}^2\,.
\end{align}
As $\psi_{n,\ell}^* \in \Phi_n^* \subseteq \Phi_n$ we have $(\psi_{n,\ell}^*)^{[K_n^*]} = \psi_{n,\ell}^*$ and $\langle g,\psi_{n,\ell}^*\rangle_{p_X} = \langle g^{[K_n^*]} , \psi_{n,\ell}^*\rangle_{p_X}$ and the first term in (\ref{eq:risk1}) equals
\begin{align*}
\sum_{\ell=1}^{J_n^*} \mbox{var} \Big(\frac1m \sum_{j=1}^m \psi_{n,\ell}^*(X_j)\, g^{[K_n^*]}(X_j)\Big) & \, \leq \, \frac1m \, \sum_{\ell=1}^{J_n^*} \mathbb{E} \big|\psi_{n,\ell}^*(X_1)\big|^2 \big|g^{[K_n^*]}(X_1)\big|^2 \\ & \, \leq \, m^{-1} \rho^{-1} \big\{1 + d_n J_n (1+\pi^2/6)\big\} \cdot \|g\|_{p_X}^2 \\
& \, \leq \, m^{-1} \rho^{-1} \big\{1 + d_n J_n (1+\pi^2/6)\big\} \cdot C^2 \, d_n^2 \,,
\end{align*}
using Lemma \ref{L:ON} and (\ref{eq:Hoelder}). As the $\eta_{n,k}$ are i.i.d.~standard Gaussian, the second term in (\ref{eq:risk1}) turns out to be the sum of
\begin{align*}
\frac{\sigma^2}m & \, \sum_{k'=1}^{K_n^*} \mathbb{E} \Big|\sum_{k=1}^{K_n^*} \langle \psi_{n,\ell}^*, \psi_{n,k}\rangle_{p_X} (\widehat{M}_{n,1})_{k,k'}^{1/2}\Big|^2 \\ & \, = \, \frac{\sigma^2}m \, \sum_{k,k'=1}^{K_n^*} \langle \psi_{n,\ell}^* , \psi_{n,k}\rangle_{p_X} \langle \psi_{n,\ell}^*,\psi_{n,k'}\rangle_{p_X} \, \big(\mathbb{E} \widehat{M}_{n,1}\big)_{k,k'} \\ &  \, = \, \frac{\sigma^2}m \sum_{k=1}^{K_n^*} \langle \psi_{n,\ell}^* , \psi_{n,k}\rangle_{p_X}^2 \\
& \, \leq \, \sigma^2 \, m^{-1}\,,
\end{align*}
over $\ell=1,\ldots,J_n^*$. Finally, concerning the third term in (\ref{eq:risk1}),
$$ \sum_{\ell>J_n^*} \langle g , \psi_{n,\ell}^*\rangle_{p_X}^2 \, = \, \big\|g - g^{[J_n^*,*]}\big\|_{p_X}^2 \, \leq \, \rho^{-1} \cdot d_n \cdot C^2 \cdot J_n^{-2\beta}\,, $$
where $g^{[J_n^*,*]}$ stands for the orthogonal projection of $g$ onto $\Phi_n^*$ and the upper bound (\ref{eq:2.2}) is applied when $J_n$ replaces $K_n$.

Turning to the statistical risk of $\widehat{g}_{n,2}$ we may employ the analogous upper bound where the first term in (\ref{eq:risk1}) vanishes, $\widehat{M}_{n,1}^{1/2}$ shall be replaced by $I_{K_n^*}$ with respect to the second term, and $m$ must be replaced by $n-m$. \end{proof}

%\smallskip

%%%%%%%%%%%%%%%%%%%%%%%%%%%%%%%%%%%%%%%%%%%%%%
%% Single Appendix:                         %%
%%%%%%%%%%%%%%%%%%%%%%%%%%%%%%%%%%%%%%%%%%%%%%
%\begin{appendix}
%\section*{???}%% if no title is needed, leave empty \section*{}.
%\end{appendix}
%%%%%%%%%%%%%%%%%%%%%%%%%%%%%%%%%%%%%%%%%%%%%%
%% Multiple Appendixes:                     %%
%%%%%%%%%%%%%%%%%%%%%%%%%%%%%%%%%%%%%%%%%%%%%%
%\begin{appendix}
%\section{???}
%
%\section{???}
%
%\end{appendix}

%%%%%%%%%%%%%%%%%%%%%%%%%%%%%%%%%%%%%%%%%%%%%%
%% Support information, if any,             %%
%% should be provided in the                %%
%% Acknowledgements section.                %%
%%%%%%%%%%%%%%%%%%%%%%%%%%%%%%%%%%%%%%%%%%%%%%
%\begin{acks}[Acknowledgments]
% The authors would like to thank ...
%\end{acks}
%%%%%%%%%%%%%%%%%%%%%%%%%%%%%%%%%%%%%%%%%%%%%%
%% Funding information, if any,             %%
%% should be provided in the                %%
%% funding section.                         %%
%%%%%%%%%%%%%%%%%%%%%%%%%%%%%%%%%%%%%%%%%%%%%%
\begin{funding}
M.~Jirak is funded by the Austrian Science Fund (FWF), Project 5485. A.~Meister and A.~Rohde are supported by the Research Unit 5381 (DFG), ME 2114/5-1 and RO 3766/8-1, respectively.
\end{funding}

%%%%%%%%%%%%%%%%%%%%%%%%%%%%%%%%%%%%%%%%%%%%%%
%% Supplementary Material, including data   %%
%% sets and code, should be provided in     %%
%% {supplement} environment with title      %%
%% and short description. It cannot be      %%
%% available exclusively as external link.  %%
%% All Supplementary Material must be       %%
%% available to the reader on Project       %%
%% Euclid with the published article.       %%
%%%%%%%%%%%%%%%%%%%%%%%%%%%%%%%%%%%%%%%%%%%%%%
%\begin{supplement}
%\stitle{???}
%\sdescription{???.}
%\end{supplement}

%%  MR numbers will be added by VTeX.                      %%
%%                                                         %%
%%  Use \cite{...} to cite references in text.             %%
%%                                                         %%
%%%%%%%%%%%%%%%%%%%%%%%%%%%%%%%%%%%%%%%%%%%%%%%%%%%%%%%%%%%%%

%% if your bibliography is in bibtex format, uncomment commands:
\bibliographystyle{imsart-nameyear} % Style BST file (imsart-number.bst or imsart-nameyear.bst)
\bibliography{Referenzen}       % Bibliography file (usually '*.bib')

%% or include bibliography directly:
% \begin{thebibliography}{}
% \bibitem{b1}
% \end{thebibliography}

\end{document}